\documentclass[11pt]{article}
\usepackage{graphicx}
\usepackage{amsmath,amsxtra,amssymb,latexsym, amscd,amsthm}

\usepackage[colorlinks=true,citecolor=black,linkcolor=black,urlcolor=blue]{hyperref}

\theoremstyle{plain}
\newtheorem{thm}{Theorem}[section]

\newtheorem{lem}[thm]{Lemma}

\theoremstyle{definition}

\theoremstyle{remark}
\newtheorem{rmk}{Remark}

\numberwithin{equation}{section}
\numberwithin{figure}{section}
\numberwithin{table}{section}

\newcommand{\M}{\operatorname{M}}

\begin{document}
\setlength{\baselineskip}{13truept}
\title{\bf Enumeration of Hybrid Domino-Lozenge Tilings II: Quasi-octagonal regions}

\author{TRI LAI\\Institute for Mathematics and Its Applications\\University of Minnesota\\Minneapolis, MN 55455\\tlai@umn.edu}
\date{\today}
\maketitle
\begin{abstract}
We use the subgraph replacement method to prove a simple product formula for the tilings of an  8-vertex counterpart of Propp's quasi-hexagons (Problem 16 in \textit{New Perspectives in Geometric Combinatorics}, Cambridge University Press, 1999), called quasi-octagon.

\bigskip\noindent \textbf{Keywords:} perfect matchings, tilings, dual graphs, Aztec diamonds, Aztec rectangles, quasi-octagons.
\end{abstract}


\section{Introduction}

We are interested in regions on the square lattice with southwest-to-northeast diagonals drawn in. In 1996, Douglas \cite{Doug} proved a conjecture posed by Propp on the number of tilings of a certain family of regions in the square lattice with every second \textit{diagonal}\footnote{From now on,  the term ``diagonal" will be used to mean ``southwest-to-northeast drawn-in diagonal".}. In particular, Douglas showed that the region of order $n$ (shown in Figure \ref{douglas}) has $2^{2n(n+1)}$ tilings. In 1999, Propp listed 32 open problems in enumeration of perfect matchings and tilings  in his survey paper \cite{Propp}. Problem 16 on the list asks for a formula for the number of tilings of a certain quasi-hexagonal region on the square lattice with every third diagonal. The problem has been solved and generalized by author of the paper (see \cite{Tri})  for a certain quasi-hexagons in which distances\footnote{The unit here is  $\sqrt{2}/2$.} between two successive diagonals are arbitrary. The method, subgraph replacement method, provided also a generalization of Douglas' result above.

\begin{figure}\centering
\begin{picture}(0,0)%
\includegraphics{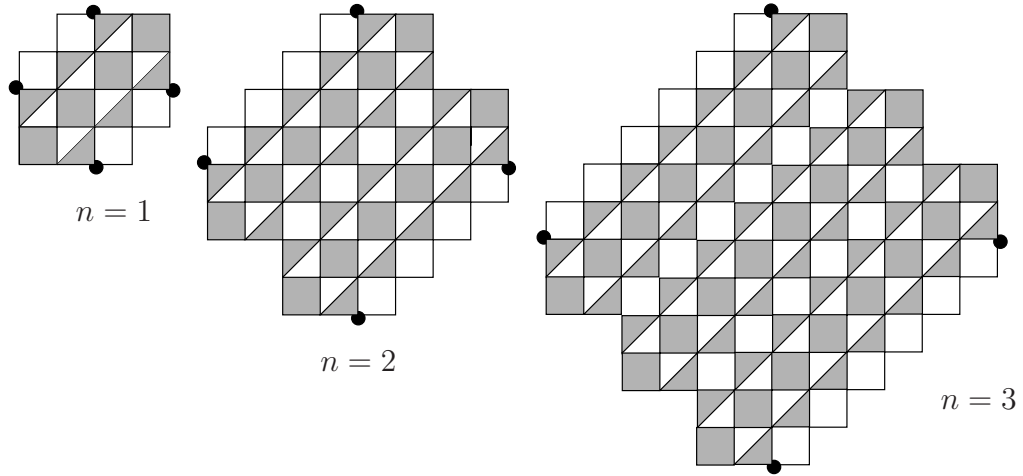}%
\end{picture}%
\setlength{\unitlength}{3947sp}%
\begingroup\makeatletter\ifx\SetFigFont\undefined%
\gdef\SetFigFont#1#2#3#4#5{%
  \reset@font\fontsize{#1}{#2pt}%
  \fontfamily{#3}\fontseries{#4}\fontshape{#5}%
  \selectfont}%
\fi\endgroup%
\begin{picture}(6283,2966)(402,-2292)
\put(2363,-1641){\makebox(0,0)[lb]{\smash{{\SetFigFont{12}{14.4}{\familydefault}{\mddefault}{\updefault}{$n=2$}%
}}}}
\put(6260,-1878){\makebox(0,0)[lb]{\smash{{\SetFigFont{12}{14.4}{\familydefault}{\mddefault}{\updefault}{$n=3$}%
}}}}
\put(828,-696){\makebox(0,0)[lb]{\smash{{\SetFigFont{12}{14.4}{\familydefault}{\mddefault}{\updefault}{$n=1$}%
}}}}
\end{picture}%
\caption{The Douglas' regions of order $n=1$, $n=2$ and $n=3$.}
\label{douglas}
\end{figure}

In this paper, we use the subgraph replacement method to enumerate tilings of a new family of 8-vertex regions that are inspired by Propp's quasi-hexagons (see \cite{Propp}, \cite{Tri}) and defined in the next two paragraphs.

\begin{figure}\centering
\begin{picture}(0,0)%
\includegraphics{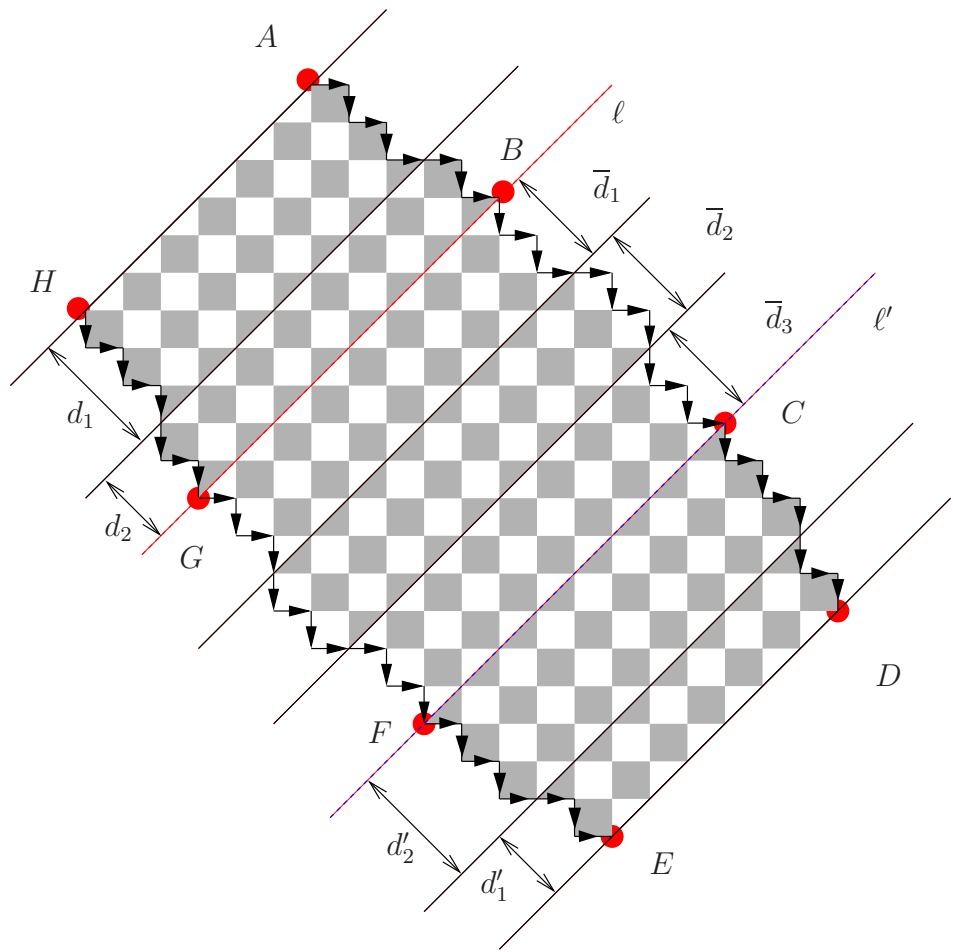}%
\end{picture}%
\setlength{\unitlength}{3947sp}%
\begingroup\makeatletter\ifx\SetFigFont\undefined%
\gdef\SetFigFont#1#2#3#4#5{%
  \reset@font\fontsize{#1}{#2pt}%
  \fontfamily{#3}\fontseries{#4}\fontshape{#5}%
  \selectfont}%
\fi\endgroup%
\begin{picture}(5929,5930)(698,-6496)
\put(2245,-814){\makebox(0,0)[lb]{\smash{{\SetFigFont{12}{14.4}{\rmdefault}{\mddefault}{\updefault}{$A$}%
}}}}
\put(3781,-1523){\makebox(0,0)[lb]{\smash{{\SetFigFont{12}{14.4}{\rmdefault}{\mddefault}{\updefault}{$B$}%
}}}}
\put(5552,-3177){\makebox(0,0)[lb]{\smash{{\SetFigFont{12}{14.4}{\rmdefault}{\mddefault}{\updefault}{$C$}%
}}}}
\put(6143,-4830){\makebox(0,0)[lb]{\smash{{\SetFigFont{12}{14.4}{\rmdefault}{\mddefault}{\updefault}{$D$}%
}}}}
\put(4725,-6011){\makebox(0,0)[lb]{\smash{{\SetFigFont{12}{14.4}{\rmdefault}{\mddefault}{\updefault}{$E$}%
}}}}
\put(2954,-5185){\makebox(0,0)[lb]{\smash{{\SetFigFont{12}{14.4}{\rmdefault}{\mddefault}{\updefault}{$F$}%
}}}}
\put(1773,-4088){\makebox(0,0)[lb]{\smash{{\SetFigFont{12}{14.4}{\rmdefault}{\mddefault}{\updefault}{$G$}%
}}}}
\put(828,-2350){\makebox(0,0)[lb]{\smash{{\SetFigFont{12}{14.4}{\rmdefault}{\mddefault}{\updefault}{$H$}%
}}}}
\put(4489,-1287){\makebox(0,0)[lb]{\smash{{\SetFigFont{12}{14.4}{\rmdefault}{\mddefault}{\updefault}{$\ell$}%
}}}}
\put(1064,-3177){\makebox(0,0)[lb]{\smash{{\SetFigFont{12}{14.4}{\rmdefault}{\mddefault}{\updefault}{$d_1$}%
}}}}
\put(1301,-3885){\makebox(0,0)[lb]{\smash{{\SetFigFont{12}{14.4}{\rmdefault}{\mddefault}{\updefault}{$d_2$}%
}}}}
\put(4371,-1759){\makebox(0,0)[lb]{\smash{{\SetFigFont{12}{14.4}{\rmdefault}{\mddefault}{\updefault}{$\overline{d}_1$}%
}}}}
\put(5080,-1996){\makebox(0,0)[lb]{\smash{{\SetFigFont{12}{14.4}{\rmdefault}{\mddefault}{\updefault}{$\overline{d}_2$}%
}}}}
\put(5454,-2581){\makebox(0,0)[lb]{\smash{{\SetFigFont{12}{14.4}{\rmdefault}{\mddefault}{\updefault}{$\overline{d}_3$}%
}}}}
\put(3072,-5893){\makebox(0,0)[lb]{\smash{{\SetFigFont{12}{14.4}{\rmdefault}{\mddefault}{\updefault}{$d'_2$}%
}}}}
\put(3662,-6129){\makebox(0,0)[lb]{\smash{{\SetFigFont{12}{14.4}{\rmdefault}{\mddefault}{\updefault}{$d'_1$}%
}}}}
\put(6129,-2589){\makebox(0,0)[lb]{\smash{{\SetFigFont{12}{14.4}{\rmdefault}{\mddefault}{\updefault}{$\ell'$}%
}}}}
\end{picture}%
\caption{The quasi-octagon $\mathcal{O}_{6}(5,3;\ 4,4,4;\ 3,5)$.}
\label{octagonfig}
\end{figure}

We consider two distinguished diagonals $\ell$ and $\ell'$, so that $\ell'$ is below $\ell$.  Draw in $k$ diagonals above $\ell$, with the distances between consecutive ones, starting from top, being $d_1,\dotsc,d_k$; and $l$ diagonals below $\ell'$,  with the distances between consecutive ones, starting from the bottom, being $d'_1,\dotsc,d'_l$. Draw in also $t-1$ additional diagonals between $\ell$ and $\ell'$, so that the distances between successive ones (starting from top) are $\overline{d}_1,\dotsc,\overline{d}_t$ (see Figure \ref{octagonfig}). To make the setup of diagonals well-defined, we assume that the distance between $\ell$ and $\ell'$ is at least $t$.

Next, we color the resulting dissection of the square lattice black and white, so that any two fundamental regions sharing an edge have opposite colors. Without loss of generality, we can assume that the triangles pointing toward $\ell$ and  having bases on the top diagonal are white. Let $A$ be a lattice point on the top diagonal.  We draw a lattice path with unit steps south or east  from $A$ so that at each step the black fundamental region is on the right. The lattice path meets $\ell$ at a lattice point $B$. Continue from $B$ to a vertex $C$ on $\ell'$ in same fashion, with the difference that the black fundamental region is now on the left at each step. Finally, we go from $C$ to a vertex $D$ on the bottom diagonal in the same way with the black fundamental region is on the right at each step. The described path from $A$ to $D$ is the northeastern boundary of the region.

Pick a lattice point $H$ on the top diagonal to the left of $A$ so that the Euclidean distance between them is $|AH|=a\sqrt{2}$. The southwestern boundary is defined analogously, going from $H$ to a point $G\in \ell$, then to a point $F\in \ell'$, and to a point $E$ on the bottom diagonal. One can see that the southwestern boundary is obtained by reflecting the northeastern one about the perpendicular bisector of  the segment $AH$.  The segments $AH$ and $DE$ complete the boundary of the region, which we denote by  $\mathcal{O}_{a}(d_1,\dotsc,d_k;\ \overline{d}_1,\dotsc, \overline{d}_t;\ d'_1,\dotsc,d'_l)$ (see Figure \ref{octagonfig} for an example) and call a \textit{quasi-octagon}.

Call the fundamental regions inside $\mathcal{O}_{a}(d_1,\dotsc,d_k;\ \overline{d}_1,\dotsc, \overline{d}_t;\ d'_1,\dotsc,d'_l)$ \textit{cells}, and call them black or white according to the coloring described above.  Note that there are two kinds of cells, square and triangular. The latter in turn come in two orientations: they may point towards $\ell'$ or away from $\ell'$.  A cell is called \textit{regular} if it is a square cell or a triangular cell pointing away from $\ell'$.

A \textit{row of cells} consists of all the triangular cells of a given color with bases resting on a fixed lattice diagonal, or of all the square cells (of a given color) passed through by a fixed lattice diagonal.

\begin{rmk}
If the triangular cells running along the bottom diagonal  of a quasi-octagon are black, then we can not cover these cells by disjoint tiles, and the region has no tilings. \textit{Therefore, from now on, we assume that the triangular cells running along the bottom diagonal are white.} This is equivalent to the the fact that the last step of the southwestern boundary is an east step.
\end{rmk}


The \textit{upper, lower}, and \textit{middle parts} of the region are defined to be the portions above $\ell$, below $\ell'$, and  between $\ell$ and $\ell'$ of the region.  We define \textit{the upper} and \textit{lower heights} of our region to be the numbers of rows of black regular cells in the upper and lower parts.  The \textit{middle height} is the number of rows of \textit{white} regular cells in the middle part. Denote by $h_1(\mathcal{O}),$ $h_2(\mathcal{O})$ and $h_3(\mathcal{O})$ the upper, middle and lower heights of a quasi-octagon $\mathcal{O}$, respectively. Define also \textit{the upper} and\textit{ lower widths} of $\mathcal{O}$ to be $w_1(\mathcal{O})=|BG|/\sqrt{2}$ and $w_2(\mathcal{O})=|CF|/\sqrt{2}$, where $|BG|$ and $|CF|$ are the Euclidean lengths of the segments $BG$ and $CF$, respectively. For example, the quasi-octagon in Figure \ref{octagonfig} has the upper, middle and lower heights 5,6,5, respectively, and has the upper and lower widths equal to $8$.

The main result of the paper concerns the number of tilings of quasi-octagons whose upper and lower widths are equal. In this case, the number of tilings is given by a simple product formula (unfortunately, in general, a quasi-octagon  does not lead to a simple product formula).

The number of tilings of a quasi-octagon with equal widths is given by the theorem stated below.

\begin{thm}\label{octagon1}
Let $a$, $d_1,$ $\dotsc$, $d_k$; $\overline{d}_1$, $\dotsc$, $\overline{d}_t$; $d'_1$, $\dotsc$, $d'_l$ be positive integers for which the quasi-octagon $\mathcal{O}:=\mathcal{O}_{a}(d_1,\dotsc,d_k;\ \overline{d}_1,\dotsc, \overline{d}_t;\ d'_1,\dotsc,d'_l)$ has the upper, the middle, and the lower heights $h_1,h_2,h_3$, respectively, and has both upper and lower widths equal to $w$, with $w>h_1,h_2,h_3$.

$(a)$ If $h_1+h_3\not=w+h_2$, then $\mathcal{O}$ has no tilings.

$(b)$ If $h_1+h_3=w+h_2$, then the number of tilings of $\mathcal{O}$ is equal to
\begin{align}\label{mainoc2}
&2^{\mathcal{C}_1+\mathcal{C}_2+\mathcal{C}_3-h_1(2w-h_1+1)/2-h_2(2w-h_2+1)/2-h_3(2w-h_3+1)/2}\notag\\
&\times2^{\binom{h_1+2h_2+h_2}{2}-2h_2(w+h_2)-\binom{h_1+h_2}{2}-\binom{h_2+h_3}{2}}\notag\\
&\times\dfrac{\prod_{i=h_2+h_3+1}^{h_2+w}(i-1)!\prod_{i=h_1+h_2+1}^{h_2+w}(i-1)!\prod_{i=1}^{w-h_3}(i-1)!\prod_{i=1}^{w-h_1}(i-1)!}
{\prod_{i=1}^{w-h_2}(h_2+i-1)!(w-i)!},
\end{align}
where $\mathcal{C}_1$ is the numbers of black regular cells in the upper part, $\mathcal{C}_2$ is the number of  white regular cells in the middle part, and $\mathcal{C}_3$ is the number of black regular cells in the lower part of the region.
\end{thm}

We notice that if we consider the case  when $\ell$ and $\ell'$ are superimposed in the definition of a quasi-octagon,
we get a region with six vertices that is exactly a symmetric quasi-hexagon defined in \cite{Tri}. We showed in \cite{Tri} that the number of tilings of a symmetric quasi-hexagon is a power of 2 times the number of tilings of a hexagon on the triangular lattice (see Theorem 2.2 in \cite{Tri}).

We will use a result of Kattenthaler \cite{Krat} (about the number of tilings of a certain family of Aztec rectangles with holes) and several new subgraph replacement rules to prove Theorem \ref{octagon1} (see Section 3).  As mentioned before, in general, the number of tilings of a quasi-octagon is \textit{not} given by a simple product formula. However,  we  can prove a sum formula for the number of tilings of the region in the general case (see Theorem \ref{octa3}).

\section{Preliminaries}

This paper shares several preliminary results and definitions with its prequel \cite{Tri}. If ones already read \cite{Tri}, they would like to move their attention quickly to Lemma \ref{masstransform2}.

A \textit{perfect matching} of a graph $G$ is a collection of edges such that each vertex of $G$ is adjacent to precisely one edge in the collection. A perfect matching is called \textit{dimmer covering} in statistical mechanics, and also \textit{1-factor} in graph theory. The number of perfect matchings of $G$ is denoted by $\M(G)$. More generally, if the edges of $G$ have weights on them, $\M(G)$ denotes the sum of the weights of all perfect matchings of $G$, where the weight of a perfect matching is the product of the weights on its constituent edges.

Given a lattice in the plane, a (lattice) \textit{region} is a finite connected union of fundamental regions of that lattice. A \textit{tile} is the union of any two fundamental regions sharing an edge. A \textit{tiling} of the region $R$ is a covering of $R$ by tiles with no gaps or overlaps. The tilings of a region $R$ can be naturally identified with the perfect matchings of its \textit{dual graph} (i.e., the graph whose vertices are the fundamental regions of $R$, and whose edges connect two fundamental regions precisely when they share an edge). In view of this, we denote by $\M(R)$ the number of tilings of $R$.

A \textit{forced edge} of a graph $G$ is an edge contained in every perfect matching of $G$. 
Removing forced edges and the vertices incident to them does not change the number of perfect matching of a graph\footnote{From now on, whenever we remove some forced edges, we remove also the vertices incident to them.}.

We present next three basic preliminary results stated below.

\begin{figure}\centering
\begin{picture}(0,0)%
\includegraphics{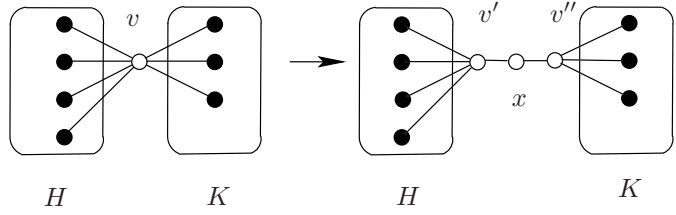}%
\end{picture}%
%
%
\setlength{\unitlength}{3947sp}%
\begingroup\makeatletter\ifx\SetFigFont\undefined%
\gdef\SetFigFont#1#2#3#4#5{%
  \reset@font\fontsize{#1}{#2pt}%
  \fontfamily{#3}\fontseries{#4}\fontshape{#5}%
  \selectfont}%
\fi\endgroup%
\begin{picture}(4188,1361)(593,-556)
\put(1336,591){\makebox(0,0)[lb]{\smash{{\SetFigFont{10}{12.0}{\familydefault}{\mddefault}{\updefault}{$v$}%
}}}}
\put(3549,621){\makebox(0,0)[lb]{\smash{{\SetFigFont{10}{12.0}{\familydefault}{\mddefault}{\updefault}{$v'$}%
}}}}
\put(3757, 89){\makebox(0,0)[lb]{\smash{{\SetFigFont{10}{12.0}{\familydefault}{\mddefault}{\updefault}{$x$}%
}}}}
\put(3999,621){\makebox(0,0)[lb]{\smash{{\SetFigFont{10}{12.0}{\familydefault}{\mddefault}{\updefault}{$v''$}%
}}}}
\put(820,-541){\makebox(0,0)[lb]{\smash{{\SetFigFont{10}{12.0}{\familydefault}{\mddefault}{\updefault}{$H$}%
}}}}
\put(1840,-535){\makebox(0,0)[lb]{\smash{{\SetFigFont{10}{12.0}{\familydefault}{\mddefault}{\updefault}{$K$}%
}}}}
\put(3031,-535){\makebox(0,0)[lb]{\smash{{\SetFigFont{10}{12.0}{\familydefault}{\mddefault}{\updefault}{$H$}%
}}}}
\put(4426,-484){\makebox(0,0)[lb]{\smash{{\SetFigFont{10}{12.0}{\familydefault}{\mddefault}{\updefault}{$K$}%
}}}}
\end{picture}%
\caption{Vertex splitting.}
\label{vertexsplitting}
\end{figure}

\begin{lem} [Vertex-Splitting Lemma,  Lemma 2.2 in \cite{Ciucu5}]\label{VS}
 Let $G$ be a graph, $v$ be a vertex of it, and denote the set of neighbors of $v$ by $N(v)$.
  For any disjoint union $N(v)=H\cup K$, let $G'$ be the graph obtained from $G\setminus v$ by including three new vertices $v'$, $v''$ and $x$ so that $N(v')=H\cup \{x\}$, $N(v'')=K\cup\{x\}$, and $N(x)=\{v',v''\}$ (see Figure \ref{vertexsplitting}). Then $\M(G)=\M(G')$.
\end{lem}

\begin{lem}[Star Lemma,  Lemma 3.2 in \cite{Tri}]\label{star}
Let $G$ be a weighted graph, and let $v$ be a vertex of~$G$. Let $G'$ be the graph obtained from $G$ by multiplying the weights of all edges that are adjacent to $v$ by $t>0$. Then $\M(G')=t\M(G)$.
\end{lem}


The following result is a generalization due to Propp of the ``urban renewal" trick first observed by Kuperberg. 

\begin{figure}\centering
\begin{picture}(0,0)%
\includegraphics{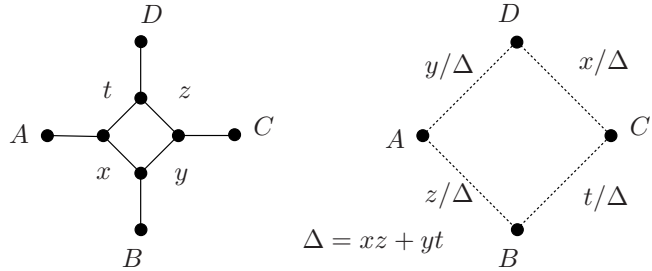}%
\end{picture}%
\setlength{\unitlength}{3947sp}%
\begingroup\makeatletter\ifx\SetFigFont\undefined%
\gdef\SetFigFont#1#2#3#4#5{%
  \reset@font\fontsize{#1}{#2}%
  \fontfamily{#3}\fontseries{#4}\fontshape{#5}%
  \selectfont}%
\fi\endgroup%
\begin{picture}(4054,1735)(340,-948)
\put(355,-156){\makebox(0,0)[lb]{\smash{{\SetFigFont{10}{12.0}{\familydefault}{\mddefault}{\updefault}{$A$}%
}}}}
\put(1064,-933){\makebox(0,0)[lb]{\smash{{\SetFigFont{10}{12.0}{\familydefault}{\mddefault}{\updefault}{$B$}%
}}}}
\put(1891,-106){\makebox(0,0)[lb]{\smash{{\SetFigFont{10}{12.0}{\familydefault}{\mddefault}{\updefault}{$C$}%
}}}}
\put(1182,603){\makebox(0,0)[lb]{\smash{{\SetFigFont{10}{12.0}{\familydefault}{\mddefault}{\updefault}{$D$}%
}}}}
\put(2717,-189){\makebox(0,0)[lb]{\smash{{\SetFigFont{10}{12.0}{\familydefault}{\mddefault}{\updefault}{$A$}%
}}}}
\put(3426,-933){\makebox(0,0)[lb]{\smash{{\SetFigFont{10}{12.0}{\familydefault}{\mddefault}{\updefault}{$B$}%
}}}}
\put(4253,-106){\makebox(0,0)[lb]{\smash{{\SetFigFont{10}{12.0}{\familydefault}{\mddefault}{\updefault}{$C$}%
}}}}
\put(3426,603){\makebox(0,0)[lb]{\smash{{\SetFigFont{10}{12.0}{\familydefault}{\mddefault}{\updefault}{$D$}%
}}}}
\put(904,-382){\makebox(0,0)[lb]{\smash{{\SetFigFont{10}{12.0}{\familydefault}{\mddefault}{\updefault}{$x$}%
}}}}
\put(1396,-388){\makebox(0,0)[lb]{\smash{{\SetFigFont{10}{12.0}{\familydefault}{\mddefault}{\updefault}{$y$}%
}}}}
\put(1418,130){\makebox(0,0)[lb]{\smash{{\SetFigFont{10}{12.0}{\familydefault}{\mddefault}{\updefault}{$z$}%
}}}}
\put(946,130){\makebox(0,0)[lb]{\smash{{\SetFigFont{10}{12.0}{\familydefault}{\mddefault}{\updefault}{$t$}%
}}}}
\put(2968,284){\makebox(0,0)[lb]{\smash{{\SetFigFont{10}{12.0}{\familydefault}{\mddefault}{\updefault}{$y/\Delta$}%
}}}}
\put(3934,311){\makebox(0,0)[lb]{\smash{{\SetFigFont{10}{12.0}{\familydefault}{\mddefault}{\updefault}{$x/\Delta$}%
}}}}
\put(3964,-544){\makebox(0,0)[lb]{\smash{{\SetFigFont{10}{12.0}{\familydefault}{\mddefault}{\updefault}{$t/\Delta$}%
}}}}
\put(2965,-526){\makebox(0,0)[lb]{\smash{{\SetFigFont{10}{12.0}{\familydefault}{\mddefault}{\updefault}{$z/\Delta$}%
}}}}
\put(2197,-817){\makebox(0,0)[lb]{\smash{{\SetFigFont{10}{12.0}{\familydefault}{\mddefault}{\updefault}{$\Delta= xz+yt$}%
}}}}
\end{picture}%
\caption{Urban renewal.}
\label{spider1}
\end{figure}

\begin{lem} [Spider Lemma]\label{spider}
 Let $G$ be a weighted graph containing the subgraph $K$ shown on the left in Figure \ref{spider1} (the labels indicate weights, unlabeled edges have weight 1). Suppose in addition that the four inner black vertices in the subgraph $K$, different from $A,B,C,D$, have no neighbors outside $K$. Let $G'$ be the graph obtained from $G$ by replacing $K$ by the graph $\overline{K}$ shown on right in Figure \ref{spider}, where the dashed lines indicate new edges, weighted as shown. Then $\M(G)=(xz+yt)\M(G')$.
%
%
%
%
\end{lem}

Consider a $(2m+1)\times(2n+1)$ rectangular chessboard and suppose the corners are black. The \textit{Aztec rectangle (graph)} $AR_{m,n}$ is the graph whose vertices are the white unit squares and whose edges connect precisely those pairs of white unit squares that are diagonally adjacent. The \textit{odd Aztec rectangle} $OR_{m,n}$ is the graph whose vertices are the black unit squares whose edges connect precisely those pairs of black unit squares that are diagonally adjacent. Figures \ref{twoAR}(a) and (b) illustrate an example of $AR_{3,5}$ and $OR_{3,5}$ respectively. If one removes the bottom row of the board $B$ and applies the same procedure the definition of $AR_{m,n}$, the resulting graph is denoted by $AR_{m-\frac12,n}$, and called a \textit{baseless Aztec rectangle} (see the graph on the right in Figure \ref{elem2} for an example with $m=3$ and $n=4$).

\begin{figure}\centering
\includegraphics[width=8cm]{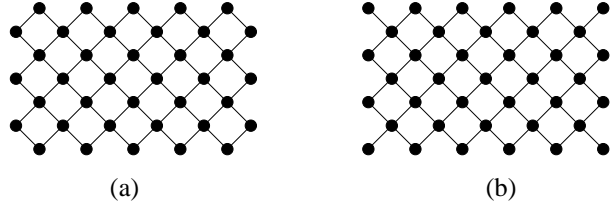}
\caption{Two kinds of Aztec rectangles.}
\label{twoAR}
\end{figure}

An \textit{induced subgraph} of a graph $G$ is a graph when vertex set is a subset $U$ of the vertex set of $G$, and whose edge set consists of all edges of $G$ with endpoints in $U$. The following lemma was proved in \cite{Tri}.

\begin{lem}[Graph Splitting Lemma, Lemma 3.6 in \cite{Tri}]\label{graphsplitting}
Let $G$ be a bipartite graph, and let $V_1$ and $V_2$ be the two vertex classes. Let $H$ be an induced subgraph of $G$.

(a) Assume that $H$ satisfies the following conditions.
\begin{enumerate}
\item[(i)] The \textit{separating condition}: there are no edges of $G$ connecting a vertex in $V(H)\cap V_1$ and a vertex in $V(G-H)$,

\item[(ii)] The\textit{ balancing condition}: $|V(H)\cap V_1|=|V(H)\cap V_2|$.
\end{enumerate}
Then
\begin{equation}\label{GS}
\M(G)=\M(H)\, \M(G-H).
\end{equation}

(b) If $H$ satisfy the separating condition and but has $|V(H)\cap V_1|>|V(H)\cap V_2|$, then $\M(G)=0$.
\end{lem}

 Let $\mathcal{D}:=D_a(d_1,\dotsc,d_k)$ be the portion of $\mathcal{O}_{a}(d_1,\dotsc,d_k;\ \overline{d}_1,\dotsc, \overline{d}_t;\ d'_1,\dotsc,d'_l)$ that is above the diagonal $\ell$, we call it a \textit{generalized Douglas region} (see \cite{Tri} for more details). The \textit{height} and the \textit{width} of $\mathcal{D}$ are defined to be to the upper height and the upper width of the quasi-octagon, denoted by $h(\mathcal{D})$ and $w(\mathcal{D})$, respectively. %
We denote by $G_a(d_1,\dotsc,d_k)$ the dual graph of $D_a(d_1,\dotsc,d_k)$.

The \textit{connected sum} $G\#G'$ of two disjoint graphs $G$ and $G'$ along the ordered sets of vertices $\{v_1,\dotsc,v_n\}\subset V(G)$ and $\{v'_1,\dotsc,v'_n\}\subset V(G')$ is the graph obtained from $G$ and $G'$ by identifying vertices $v_i$ and $v'_i$, for $i=1,\dotsc,n$. We have the following variant of Proposition 4.1 in \cite{Tri}.

\begin{figure}\centering
\includegraphics[width=10cm]{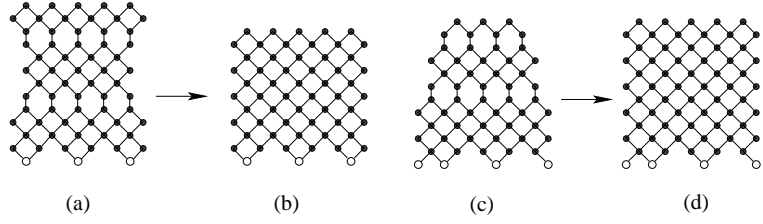}
\caption{Illustrating the transformations in Lemma \ref{masstransform2}.}
\label{massvariant}
\end{figure}

\begin{lem}[Composite Transformations]\label{masstransform2}
Assume $a$, $d_1,d_2,\dotsc,d_k$ are positive integers, so that $\mathcal{D}:=D_{a}(d_1,\dots,d_k)$ is a generalized Douglas region  having the height $h$ and the with $w$. Assume in addition that $G$ is a graph, and that  $\{v_1,v_2,\dotsc,v_{w-m}\}$ is an ordered subset of the vertex set of $G$, for some $0\leq m<w$.

 (a) If the bottom row of cells in $\mathcal{D}$ is white, then
\begin{equation}\label{composeq1}
\M(\overline{G}_a(d_1,\dotsc,d_k)\#G)=2^{\mathcal{C}-h(w+1)}\M(\overline{AR}_{h,w}\#G),
\end{equation}
where $\overline{G}_a(d_1,\dotsc,d_k)$ and $\overline{AR}_{h,w}$ are the graphs obtained from $G_a(d_1,\dotsc,d_k)$ and $AR_{h,w}$ by removing the $r_1$-st, the $r_2$-nd, $\dotsc,$ and the $r_{m}$-th vertices in their bottoms, respectively (if $m=0$, then we do not remove any vertices from the bottom of the graphs);  and where the connected sum acts on $G$ along $\{v_1,v_2,\dotsc,v_{w-m}\}$ and on $\overline{G}_a(d_1,\dotsc,d_k)$ and $\overline{AR}_{h,w}$ along their bottom vertices ordered from left to rights. See the illustration of this transformation in Figures \ref{massvariant} (a) and (b).

 (b) If the bottom row of cells in $D$ is black, then
\begin{equation}\label{composeq2}
\M(\overline{G}_a(d_1,\dotsc,d_k)\#G)=2^{\mathcal{C}-hw}\M(\overline{AR}_{h-\frac{1}{2},w-1}\#G),
\end{equation}
where $\overline{G}_a(d_1,\dotsc,d_k)$ and $\overline{AR}_{h-\frac{1}{2},w-1}$ are the graphs obtained from $G_a(d_1,\dotsc,d_k)$ and $AR_{h-\frac{1}{2},w-1}$ by removing the $r_1$-st, the $r_2$-nd, $\dotsc,$ and the $r_{m}$-th vertices in their bottoms, respectively; and where the connected sum acts on $G$ along the vertex set  $\{v_1,v_2,\dotsc,v_{w-m}\}$,  and  on $\overline{G}_a(d_1,\dotsc,d_k)$ and $\overline{AR}_{h-\frac{1}{2},w-1}$ along their bottom vertices ordered from left to right (see Figures \ref{massvariant} (c) and (d)).
\end{lem}

Since the proofs of Proposition 4.1 in \cite{Tri} and Lemma \ref{masstransform2} are identical, the proof of Lemma \ref{masstransform2} is omitted. Moreover, we get  Proposition 4.1 in \cite{Tri} from Lemma \ref{masstransform2} by specializing $m=0$.

The following lemma is a special case of the Lemma \ref{masstransform2}.

\begin{figure}\centering
\begin{picture}(0,0)%
\includegraphics{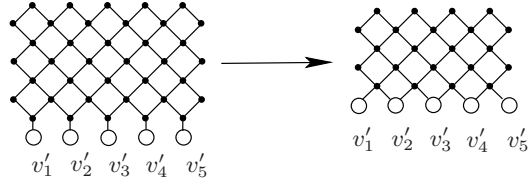}%
\end{picture}
\setlength{\unitlength}{3947sp}%
\begingroup\makeatletter\ifx\SetFigFont\undefined%
\gdef\SetFigFont#1#2#3#4#5{%
  \reset@font\fontsize{#1}{#2pt}%
  \fontfamily{#3}\fontseries{#4}\fontshape{#5}%
  \selectfont}%
\fi\endgroup%
\begin{picture}(3276,1091)(2935,-582)
\put(6004,-403){\makebox(0,0)[lb]{\smash{{\SetFigFont{9}{10.8}{\rmdefault}{\mddefault}{\updefault}{$v'_5$}%
}}}}
\put(5749,-406){\makebox(0,0)[lb]{\smash{{\SetFigFont{9}{10.8}{\rmdefault}{\mddefault}{\updefault}{$v'_4$}%
}}}}
\put(5518,-400){\makebox(0,0)[lb]{\smash{{\SetFigFont{9}{10.8}{\rmdefault}{\mddefault}{\updefault}{$v'_3$}%
}}}}
\put(5278,-397){\makebox(0,0)[lb]{\smash{{\SetFigFont{9}{10.8}{\rmdefault}{\mddefault}{\updefault}{$v'_2$}%
}}}}
\put(5032,-400){\makebox(0,0)[lb]{\smash{{\SetFigFont{9}{10.8}{\rmdefault}{\mddefault}{\updefault}{$v'_1$}%
}}}}
\put(3989,-564){\makebox(0,0)[lb]{\smash{{\SetFigFont{9}{10.8}{\rmdefault}{\mddefault}{\updefault}{$v'_5$}%
}}}}
\put(3734,-567){\makebox(0,0)[lb]{\smash{{\SetFigFont{9}{10.8}{\rmdefault}{\mddefault}{\updefault}{$v'_4$}%
}}}}
\put(3503,-561){\makebox(0,0)[lb]{\smash{{\SetFigFont{9}{10.8}{\rmdefault}{\mddefault}{\updefault}{$v'_3$}%
}}}}
\put(3263,-558){\makebox(0,0)[lb]{\smash{{\SetFigFont{9}{10.8}{\rmdefault}{\mddefault}{\updefault}{$v'_2$}%
}}}}
\put(3017,-561){\makebox(0,0)[lb]{\smash{{\SetFigFont{9}{10.8}{\rmdefault}{\mddefault}{\updefault}{$v'_1$}%
}}}}
\end{picture}
\caption{Illustrating the transformation in Lemma \ref{T1}.}
\label{elem2}
\end{figure}

\begin{lem}[Lemma 3.4 in \cite{Tri}]\label{T1}
Let $G$ be a graph and let  $\{v_1,\dotsc,v_q\}$ be an ordered subset of its vertices. Then
\begin{equation}\label{T1eq1}
\M({}_|AR_{p,q}\#G)=2^p\M(AR_{p-\frac12,q-1}\#G),
\end{equation}
where ${}_|AR_{p,q}$ is the graph obtained from $AR_{p,q}$ by appending $q$ vertical edges to its bottommost vertices; and  where the connected sum acts on $G$ along $\{v_1,\dotsc,v_q\}$, and on ${}_|AR_{p,q}$ and $AR_{p-\frac12,q}$ along their $q$ bottommost vertices ordered from left to right. The transformation is illustrated in Figure \ref{elem2}.
\end{lem}

The next result is due to Cohn, Larsen and Propp (see \cite{Cohn}, Proposition 2.1; \cite{Gessel}, Lemma 2). A \textit{$(a,b)$-semihexagon} is the bottom half of a lozenge hexagon  of side-lengths $b$, $a$, $a$, $b$, $a$, $a$ (clockwise from top) on the triangular lattice.

\begin{lem}\label{Hexdent}  Label the topmost vertices of the dual graph of $(a,b)$-semihexagon from left to right by $1,\dotsc,a+b$, and the number of perfect matchings of the graph obtained from it by removing the vertices with labels in the set $\{r_1,\dotsc,r_{a}\}$ is equal to
\begin{equation}\label{hexdenteq}
\prod_{1\leq i<j\leq a}\frac{r_j-r_i}{j-i},
\end{equation}
where $1\leq r_1<\cdots<r_{a}\leq a+b$ are given integers.
\end{lem}

We conclude this section by quoting the following result of Krattenthaler.

\begin{lem}[Krattenthaler \cite{Krat}, Theorem 14]\label{schur3}
Let $m,n,c,f$ be positive integers, and $d$ be a nonnegative integer with $2n+1\leq 2m+d-1\geq n$ and $c+(2n-2m-d+1)f\leq n+1$. Let $G$ be a $(2m+d-1)\times n$ Aztec rectangle, where all the vertices on the horizontal row that is by $d\sqrt{2}/2$ units below the central row, except for the $c$-th, $(c+f)$-th, \dots, and the $(c+(2n-2m-d+1)f)$-th vertex, have been removed. Then the number o perfect matchings of $G$  equals
\begin{align}\label{Krat-formula}
&2^{\binom{2m+d}{2}+(n+1)(n-2m-d+1)}f^{m^2+(d-1)m+\binom{d}{2}+n(n-2m-d+1)}\notag\\ &\times
\frac{\prod_{i=m+1}^{n+1}(i-1)!\prod_{i=m+d+1}^{n+1}(i-1)!\prod_{i=1}^{n-m+1}(i-1)!\prod_{i=1}^{n-m-d+1}(i-1)!}{\prod_{i=1}^{2n-2m-d+2}(c+f(i-1)-1)!(n+1-c-f(i-1))!},
\end{align}
where the product $\prod_{m+d+1}^{n+1}(i-1)!$ has to be interpreted as $1$ if $n=m+d-1$, and as $0$ if $n<m+d-1$, and similarly for the other products
\end{lem}

\section{Proof of Theorem \ref{octagon1}}

Before proving Theorem \ref{octagon1}, we present several new transformations stated below.

\begin{figure}\centering
\begin{picture}(0,0)%
\includegraphics{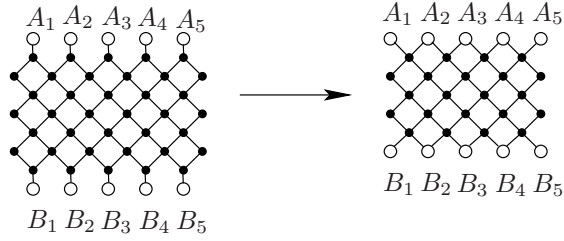}%
\end{picture}%
\setlength{\unitlength}{3947sp}%
\begingroup\makeatletter\ifx\SetFigFont\undefined%
\gdef\SetFigFont#1#2#3#4#5{%
  \reset@font\fontsize{#1}{#2pt}%
  \fontfamily{#3}\fontseries{#4}\fontshape{#5}%
  \selectfont}%
\fi\endgroup%
\begin{picture}(3419,1513)(319,-790)
\put(421,503){\makebox(0,0)[lb]{\smash{{\SetFigFont{10}{12.0}{\rmdefault}{\mddefault}{\updefault}{$A_1$}%
}}}}
\put(1362,493){\makebox(0,0)[lb]{\smash{{\SetFigFont{10}{12.0}{\rmdefault}{\mddefault}{\updefault}{$A_5$}%
}}}}
\put(652,518){\makebox(0,0)[lb]{\smash{{\SetFigFont{10}{12.0}{\rmdefault}{\mddefault}{\updefault}{$A_2$}%
}}}}
\put(1123,502){\makebox(0,0)[lb]{\smash{{\SetFigFont{10}{12.0}{\rmdefault}{\mddefault}{\updefault}{$A_4$}%
}}}}
\put(898,509){\makebox(0,0)[lb]{\smash{{\SetFigFont{10}{12.0}{\rmdefault}{\mddefault}{\updefault}{$A_3$}%
}}}}
\put(2660,532){\makebox(0,0)[lb]{\smash{{\SetFigFont{10}{12.0}{\rmdefault}{\mddefault}{\updefault}{$A_1$}%
}}}}
\put(2900,539){\makebox(0,0)[lb]{\smash{{\SetFigFont{10}{12.0}{\rmdefault}{\mddefault}{\updefault}{$A_2$}%
}}}}
\put(3133,539){\makebox(0,0)[lb]{\smash{{\SetFigFont{10}{12.0}{\rmdefault}{\mddefault}{\updefault}{$A_3$}%
}}}}
\put(3365,539){\makebox(0,0)[lb]{\smash{{\SetFigFont{10}{12.0}{\rmdefault}{\mddefault}{\updefault}{$A_4$}%
}}}}
\put(3605,539){\makebox(0,0)[lb]{\smash{{\SetFigFont{10}{12.0}{\rmdefault}{\mddefault}{\updefault}{$A_5$}%
}}}}
\put(430,-766){\makebox(0,0)[lb]{\smash{{\SetFigFont{10}{12.0}{\rmdefault}{\mddefault}{\updefault}{$B_1$}%
}}}}
\put(664,-769){\makebox(0,0)[lb]{\smash{{\SetFigFont{10}{12.0}{\rmdefault}{\mddefault}{\updefault}{$B_2$}%
}}}}
\put(892,-772){\makebox(0,0)[lb]{\smash{{\SetFigFont{10}{12.0}{\rmdefault}{\mddefault}{\updefault}{$B_3$}%
}}}}
\put(1132,-775){\makebox(0,0)[lb]{\smash{{\SetFigFont{10}{12.0}{\rmdefault}{\mddefault}{\updefault}{$B_4$}%
}}}}
\put(1369,-772){\makebox(0,0)[lb]{\smash{{\SetFigFont{10}{12.0}{\rmdefault}{\mddefault}{\updefault}{$B_5$}%
}}}}
\put(2668,-531){\makebox(0,0)[lb]{\smash{{\SetFigFont{10}{12.0}{\rmdefault}{\mddefault}{\updefault}{$B_1$}%
}}}}
\put(2902,-534){\makebox(0,0)[lb]{\smash{{\SetFigFont{10}{12.0}{\rmdefault}{\mddefault}{\updefault}{$B_2$}%
}}}}
\put(3130,-537){\makebox(0,0)[lb]{\smash{{\SetFigFont{10}{12.0}{\rmdefault}{\mddefault}{\updefault}{$B_3$}%
}}}}
\put(3370,-540){\makebox(0,0)[lb]{\smash{{\SetFigFont{10}{12.0}{\rmdefault}{\mddefault}{\updefault}{$B_4$}%
}}}}
\put(3607,-537){\makebox(0,0)[lb]{\smash{{\SetFigFont{10}{12.0}{\rmdefault}{\mddefault}{\updefault}{$B_5$}%
}}}}
\end{picture}%
\caption{Illustrating the transformation in Lemma \ref{OTrans}(a).}
\label{elem1}
\end{figure}

\begin{figure}\centering
\includegraphics[width=7cm]{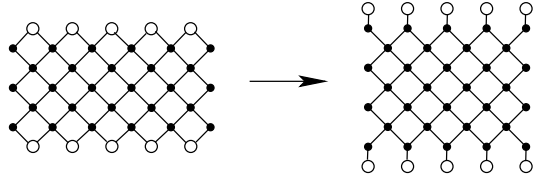}
\caption{Illustrating the transformation in Lemma \ref{OTrans}(b).}
\label{newT5}
\end{figure}

\begin{lem}\label{OTrans}
Let $G$ be a graph and let $\{v_1,\dots,v_{2q}\}$ be an ordered subset of its vertices. Then

(a) \begin{equation}\label{ot1}
 \M\left({}_|^|AR_{p,q}\#G\right)=2^{p}\M(OR_{p,q-1}\#G),
 \end{equation}
where ${}_|^|AR_{p,q}$ is the graph obtained from $AR_{p,q}$ by appending $q$ vertical edges  to its top vertices, and $q$ vertical edges to its bottom vertices; and where the connected sum acts on $G$ along $\{v_1,v_2,\dotsc,v_{2q}\}$, and on ${}_|^|AR_{p,q}$ and $OR_{p,q-1}$ along their  $q$ top vertices ordered from left to right, then along their $q$ bottom vertices ordered from left to right. See Figure \ref{elem1} for the case $p=3$ and $q=5$.

(b) \begin{equation}\label{ot2}
\M(AR_{p,q}\#G)=2^{p}\M\left({}_|^|OR_{p,q-1}\#G\right),
\end{equation}
where ${}_|^|OR_{p,q-1}$ is the graph obtained from $OR_{p,q-1}$ by appending $q$ vertical edges to its top vertices and $q$ vertical edges to its bottom vertices; and where the connected sum acts on $G$ along $\{v_1,v_2,\dotsc,v_{2q}\}$, and on $AR_{p,q}$ and ${}_|^|OR_{p,q-1}$ along their  $q$ top vertices  ordered from left to right, then along their $q$ bottom vertices ordered from left to right. See the illustration in Figure \ref{newT5} for the case $p=3$ and $q=5$.
\end{lem}

\begin{figure}\centering
\includegraphics[width=11cm]{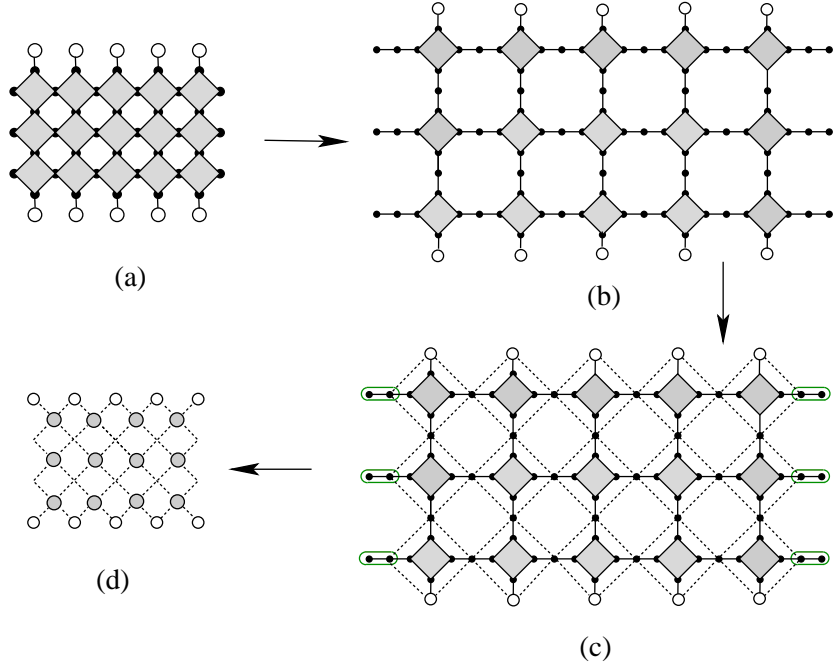}
\caption{Illustrating the proof of Lemma \ref{OTrans}(a).}
\label{proof1}
\end{figure}

\begin{proof}
We only prove part (a); and part (b) can be shown similarly.

Let $G_1$ be the graph obtained from ${}_|^|AR_{p,q}\#G$ by applying Graph-splitting Lemma \ref{VS} at all vertices that are not the end point of a vertical edge (see Figures \ref{proof1} (a) and (b) for the case $p=3$ and $q=5$). Apply Spider lemma around all $pq$ shaded diamonds in the graph $G_1$, and remove $2q$ edges adjacent to a vertex of degree 1, which are forced edges (illustrated in Figure \ref{proof1} (c); the forced edges are the circled ones). The resulting graph is isomorphic to $OR^{1/2}_{p,q-1}\#G$, where $OR^{1/2}_{p,q-1}$ is the graph obtained from $OR_{p,q-1}$ by changing all weights of edges to $1/2$ (see Figure \ref{proof1}(d); the dotted edges have weight $1/2$). By Lemmas \ref{VS} and \ref{spider}, we have
\begin{equation}\label{ot3}
\M({}_{|}^{|}AR_{p,q}\#G)=\M(G_1)=2^{pq}\M(OR^{1/2}_{p,q-1}\#G).
\end{equation}
Next, we apply Star Lemma \ref{star} with weight factor $t=2$ to all $p(q-1)$ shaded vertices of the graph $OR^{1/2}_{p,q-1}$ (see Figure \ref{proof1}(d)), the graph $OR^{1/2}_{p,q-1}\#G$ is turned into $OR_{p,q-1}\#G$. By the equality (\ref{ot3}) and Spider Lemma \ref{spider}, we get
\begin{equation}
\M({}_{|}^{|}AR_{p,q}\#G)=2^{pq}2^{-p(q-1)}\M(OR_{p,q-1}\#G),
\end{equation}
which implies (\ref{ot1}).
%
%
%
\end{proof}

Denote by $AR_{m,n}^{d}(A)$ the graph obtained from the Aztec rectangle $AR_{m,n}$ by removing all vertices at the positions in $A\subset\{1,\dotsc,n+1\}$  from the row that is by $d\sqrt{2}/2$ units below the central row (see Figure \ref{massivehole}(c) for an example).

\begin{lem}\label{massholelem}
Let $a$, $b$, $c$, $d$, $a'$, $b'$ be positive integers, so that $a+b=a'+b'$, $c<\min (a,a')$, and $d<\min(a,a')$.
Let $H_{a,b,c,d}^{a',b'}:=H_1\# H_2$, where $H_1$ is the dual graph of a hexagon of sides $b$, $a-d$, $d$, $a+b-c-d$, $c$, $a-c$, and  $H_2$ is the dual graph of a hexagon  of sides $a+b-c-d$, $d$, $a'-d$, $b'$, $a'-c$, $c$  (in cyclic order, starting from the north side); and where the connected sum acts on $H_1$ along its $a+b-c-d$ bottom vertices ordered  from left to right, and on $H_2$ along its $a+b-c-d$ top vertices ordered from left to right (see Figure \ref{massivehole}(a) for the case $a=7$, $b=3$, $c=3$, $d=4$, $a'=8$, $b'=2$). Then
\begin{equation}\label{combine}
\M(H_{a,b,c,d}^{a',b'})=2^{-a(a-1)/2-a'(a'-1)/2}\M(AR_{a+a'-1,a+b-1}^{a-a'}(A)),
\end{equation}
where $A=\{1,\dots,c\}\cup\{a+b-d+1,\dots,a+b\}$.
\end{lem}

\begin{figure}\centering
\includegraphics[width=10cm]{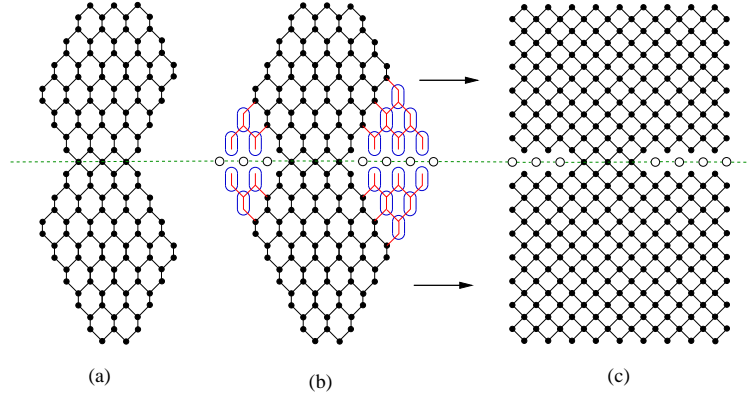}
\caption{Illustrating the proof of Lemma \ref{massholelem}.}
\label{massivehole}
\end{figure}

\begin{proof}
Consider the graph  $\widetilde{G}:=S_1\# S_2$, where $S_1$ is the graph obtained from the dual graph of the $(a,b)$-semihexagon by removing the $c$ leftmost and the $d$ rightmost bottom vertices, and  $S_2$ is the graph obtained from the dual graph of the $(a',b')$-semihexagon by removing the $c$ leftmost and the $d$ rightmost bottom vertices (see Figure \ref{massivehole}(b)), and where the connected sum acts on $S_1$ and $S_2$ along their bottom vertices ordered from left to right.  Since $H_{a,b,c,d}^{a',b'}$ is  obtained from $\widetilde{G}$ by removing vertical forced edges (the circled edges in Figure \ref{massivehole}(b)), $\M(H_{a,b,c,d}^{a',b'})=\M(\widetilde{G}).$

Apply the transformation in Lemma \ref{masstransform2}(b)  to $S_1$, we replace this graph by the graph $AR_{a-\frac{1}{2},a+b}$ with the $c$ leftmost and the $d$ rightmost bottom vertices removed. Apply this transformation one more time to $S_2$; then $S_2$ gets transformed into  the graph $AR_{a'-\frac{1}{2},a'+b'}$ with the $c$ leftmost and the $d$ rightmost bottom vertices removed.  This way, the graph $S_1\#S_2$ gets transformed precisely into the graph on the right hand side of (\ref{combine}) (see Figure \ref{massivehole}(c)). Then the lemma follows from Lemmas  \ref{masstransform2} and \ref{schur3}.
\end{proof}

\medskip

We introduce several new terminology and notations as follows.

\medskip

We divide the family of quasi-octagons into four subfamilies, based on the color of the triangles running right above $\ell$ and the color of the triangles  running right below $\ell'$. In particular, \textit{type-1 octagons} have black triangles running right above $\ell$ and right below $\ell'$; \textit{type-2 octagons} have those triangles white; \textit{type-3 octagons} have the triangles right above $\ell$ black and  the triangles right below $\ell'$ white; and  \textit{type-4 octagons} have white triangles right above $\ell$, and black triangles right below $\ell'$. To specify the type of an octagon, we denote by
\[\mathcal{O}^{(k)}_{a}(d_1,\dotsc,d_k;\ \overline{d}_1,\dotsc, \overline{d}_t;\ d'_1,\dotsc,d'_l)\]
the type-$k$ quasi-octagon with corresponding parameters (i.e. we add the superscript $k$ to the original denotation of the quasi-octagon). The
dual graph of the region is denoted by \[G^{(k)}_{a}(d_1, \dotsc, d_k;\ \overline{d}_1, \dotsc, \overline{d}_t;\ d'_1, \dotsc, d'_l).\]

The diagonals divide the middle part of a  quasi-octagon into of $t$ parts, called \textit{middle layers}. We define the \textit{height} of a middle layer to be the number of rows of white regular cells in the layer, the \textit{width} of the layer to be the number of cells on each of those rows. Assume that the $i$-th middle layer has the height $a_i$ and the width $b_i$, for $1\leq i \leq t$; then the middle height of the quasi-octagon is $\sum_{i=1}^ta_i$. Moreover, one can see that $|b_i-b_{i+1}|=1$, for any $i=1,\dotsc,t-1$. A term $b_i$ satisfying $b_i=b_{i-1}-1=b_{i+1}-1$ (resp., $b_i=b_{i-1}+1=b_{i+1}+1$) is called a \textit{concave term} (resp.,  a \textit{convex term}) of the sequence $\{b_j\}_{j=1}^t$, for $i=2,\dotsc,t-1$.

\medskip

We are now ready to prove Theorem \ref{octagon1}.

\medskip

\textbf{Out line of the proof:}
\begin{itemize}
\item We consider 4 cases, based on the type of the quasi-octagon. We prove in detail the case of type-1 quasi-octagons,
and the other cases can be implied from this case.
 \item The proof the theorem for the type-1 quasi-octagon \[\mathcal{O}^{(1)}_{a}(d_1,\dotsc,d_k;\ \overline{d}_1,\dotsc, \overline{d}_t;\ d'_1,\dotsc,d'_l).\] can be divided into 3 steps:
 \begin{itemize}
 \item Simplifying to the case when $k=l=1$.
 \item Simplifying further to the case when $k=l=t=1$.
  \item Proving the statement for $k=l=t=1$.
 \end{itemize}
 \end{itemize}

\medskip

\begin{proof}[Proof of Theorem \ref{octagon1}]
Without loss of generality, we can assume that $h_1\geq h_3$ (otherwise, we reflect the region about $\ell$ and get a new quasi-octagon with the upper hight larger than the lower height).

Denote by $\mathcal{P}(c,f,m,d,n)$ the expression (\ref{Krat-formula}) in Lemma \ref{schur3}. The statement in part (b) of the  theorem is equivalent to
\begin{align}\label{mainoc'}
\M(\mathcal{O})&=2^{\mathcal{C}_1+\mathcal{C}_2+\mathcal{C}_3-h_1(2w-h_1+1)/2-h_2(2w-h_2+1)/2-h_3(2w-h_3+1)/2}\notag\\
&\times2^{-\binom{h_1+h_2}{2}-\binom{h_2+h_3}{2}}\mathcal{P}(h_2+1,1,h_2+h_3,h_1-h_3,w+h_2-1).
\end{align}
We have four cases to distinguish, based on the type of  the quasi-octagon.

\medskip

\quad\textit{Case 1. $\mathcal{O}$ is of type 1.}

\medskip

\begin{figure}\centering
\includegraphics[width=6cm]{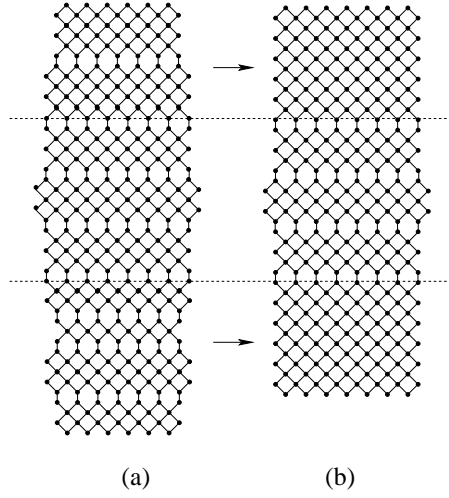}
\caption{Application of the transformation in Lemma \ref{masstransform2} to the upper and the lower parts of the dual graph of a quasi-octagon.}
\label{octa1a}
\end{figure}

STEP 1. \textit{Simplifying to the case when $k=l=1$.}\\

Let $G$ be the dual graph of the type-1 quasi-octagon
\[\mathcal{O}:=\mathcal{O}^{(1)}_{a}(d_1,\dotsc,d_k;\ \overline{d}_1,\dotsc, \overline{d}_t;\ d'_1,\dotsc,d'_l).\]
Apply the  composite transformation in Lemma \ref{masstransform2}(b) with $m=0$, separately to the portions of  $G$  corresponding to the parts above $\ell$ and below $\ell'$ of the region $\mathcal{O}$ (which we call the \textit{upper} and \textit{lower parts} of the dual graph $G$; similarly, the \textit{middle part} of $G$ corresponds to the part between $\ell$ and $\ell'$ of $\mathcal{O}$). We replace the upper part of $G$ by the graph $AR_{h_1-\frac{1}{2},w-1}$, and the lower part by the graph $AR_{h_2-\frac{1}{2},w-1}$ flipped over its base.  This way, $G$ gets transformed into the dual graph $G'$ of the type-1 quasi-octagon
\begin{equation}
\overline{\mathcal{O}}=\mathcal{O}^{(1)}_{w-1}(2h_1-1;\ \overline{d}_1,\dotsc, \overline{d}_t; 2h_3-1),
\end{equation}
and by Lemma \ref{masstransform2}, we obtain
\begin{equation}\label{octaeq1}
\M(\mathcal{O})=2^{\mathcal{C}_1-h_1w+\mathcal{C}_3-h_3w}\M(\overline{\mathcal{O}})
\end{equation}
(see Figure \ref{octa1a} for an example).

It is easy to check that $\mathcal{O}$ and $\overline{\mathcal{O}}$ have the same heights, the same widths, and the same middle part. This  and the equality (\ref{octaeq1}) imply that part (a) is true for $\mathcal{O}$ if and only if it is true for $\overline{\mathcal{O}}$. Suppose now that $h_1+h_3=h_2+w$,  the number of black regular cells in the upper part of $\overline{\mathcal{O}}$ is $\overline{\mathcal{C}}_1=h_1w$, and the number of black regular cells in the lower part of $\overline{\mathcal{O}}$ is $\overline{\mathcal{C}}_3=h_3w$ (note that we have $\mathcal{C}_2$= $\overline{\mathcal{C}}_2$ since two regions have the same middle part). Therefore, by (\ref{octaeq1}) again, we have (\ref{mainoc'}) is equivalent to
\begin{align}
\M(\overline{\mathcal{O}})&=2^{\overline{\mathcal{C}}_1+\overline{\mathcal{C}}_2+\overline{\mathcal{C}}_3-h_1(2w-h_1+1)/2-h_2(2w-h_2+1)/2-h_3(2w-h_3+1)/2}\notag\\
&\times2^{-\binom{h_1+h_2}{2}-\binom{h_2+h_3}{2}}\mathcal{P}(h_2+1,1,h_2+h_3,h_1-h_3,w+h_2-1).
\end{align}
It means that the statement in part (b) is true for the region $\mathcal{O}$ if and only if it is true for the region $\overline{\mathcal{O}}$ that has only one layer in its upper part, and only one layer in its lower part. Therefore, without of loss
 generality, we can assume that our quasi-octagon $\mathcal{O}$ has $k=l=1$.

 \medskip

\textit{STEP 2. Simplifying further to the case when $k=l=t=1$.}\\
  Assume that the $i$-th middle layer of $\mathcal{O}$ has the height $a_i$ and the width $b_i$, for $i=1,2,\dotsc,t$. Since $\mathcal{O}$ is a type-1 quasi-octagon, we have $b_1=|BG|=w=|CF|=b_t$ and
\begin{equation}\label{octaeq2}
\sum_{i=1}^{t-1}(b_i-b_{i+1})=b_1-b_t=0.
\end{equation}
Since each term of the sum on the left hand side of (\ref{octaeq2}) is either $1$ or $-1$, the numbers of $1$'s and $-1$'s are equal. It implies that $t-1$ is even, or $t$ is odd.

\medskip
We now assume that $\mathcal{O}$ has $t\geq 3$ middle layers, i.e.
\[\mathcal{O}:=\mathcal{O}^{(1)}_{w-1}(2h_1-1;\ \overline{d}_1,\dotsc, \overline{d}_t; 2h_3-1),\]
for some odd $t\geq 3$. Next, we show a way to construct a type-1 quasi-octagon $\mathcal{O}'$ having $t-2$ middle layers so that the statement of the theorem is true for $\mathcal{O}$ if and only if it is true for $\mathcal{O}'$.

\begin{figure}\centering
\includegraphics[width=12cm]{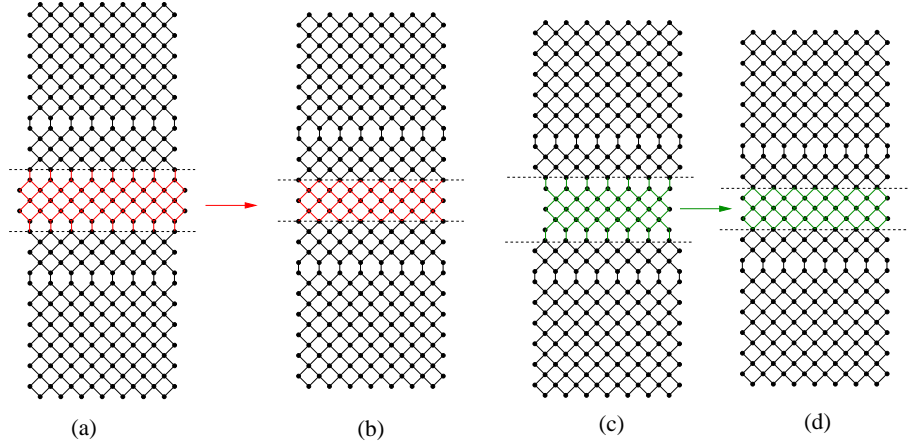}
\caption{Removing convex and concave terms from the sequence the widths of the middle layers.}
\label{octa4}
\end{figure}

 We can find two consecutive terms having opposite signs in the sequence $\{(b_{i}-b_{i+1})\}_{i=1}^{t-1}$ (otherwise, the terms are all $1$ or all $-1$; so the sum of them is different from $0$, a contradiction to (\ref{octaeq2})). Assume that $(b_{j-1}-b_{j})$ and $(b_{j}-b_{j+1})$ are such two terms, so $b_{j}$ is a convex term or a concave term of the sequence $\{b_i\}_{i=1}^{t}$.

Suppose first that $b_{j}$ is a convex term, i.e. $b_{i-1}=b_j-1=b_{j+1}$. Let $B$ be graph obtained from the dual graph of the $j$-th middle layer by appending vertical edges to its topmost and bottommost vertices. In this case, $B$  is isomorphic to the graph ${}_{|}^{|}AR_{a_{j},b_{j}}$ (see the graph between two dotted lines in Figure \ref{octa4}(a)). Apply the transformation in Lemma  \ref{OTrans}(1) to replace $B$ by $OR_{a_j,b_j-1}$. This ways, the dual graph $G$ of $\mathcal{O}$ is transformed into the dual graph $G''$ of the type-1 quasi-octagon
\begin{equation}\label{newoc}
\mathcal{O}':=\mathcal{O}^{(1)}_{w-1}(2h_1-1; d_1,\dotsc,d_{j-2}, d_{j-1}+d_{j}+d_{j+1},d_{j+2},\dotsc, d_{t}; 2h_3-1)
\end{equation}
having $t-2$ middle layers (see Figures \ref{octa4}(a) and (b) for an example). By Lemma \ref{OTrans}(a), we have
\begin{equation}\label{oceq5}
\M(\mathcal{O})=2^{a_{j}}\M(\mathcal{O}').
\end{equation}
Intuitively, we have just combined three middle layers (the $(j-1)$-th, the $i$-th and the $(j+1)$-th middle layers) of $\mathcal{O}$ into the $(j-1)$-th middle layer of $\mathcal{O}'$, and leave other parts of the region unchanged. The height of the $(j-1)$-th middle layer of $\mathcal{O}'$ is $a_{j-1}+a_{j}+a_{j+1}$, and the width of its is $b_{j-1}$. Thus
\begin{equation}
h_2=\sum_{i=1}^ta_i=\sum_{i=1}^{j-2}a_i+(a_{j-1}+a_{j}+a_{j+1})+\sum_{i=j+2}^{t-2}a_i=h'_2.
\end{equation}
Moreover, the two regions $\mathcal{O}$ and $\mathcal{O}'$ have the same upper and lower parts, so $h_1=h'_1$, $h_3=h'_3$, $\mathcal{C}_1=\mathcal{C}'_1$, $\mathcal{C}_3=\mathcal{C}'_3$ and $w=w'$ (the primed symbols refer to $\mathcal{O}'$ and  denote quantities corresponding to their unprimed counterparts of $\mathcal{O}$). This and the equality (\ref{oceq5}) imply that the statement in part (a) holds for $\mathcal{O}$ if and only if it holds for $\mathcal{O}'$.

Suppose that the condition $h_1+h_3=w+h_2$ in part (b) of the theorem holds. Note that the $i$-th middle layer of $\mathcal{O}$ has $a_i$ rows of $b_i+1$ white regular cells, for $i=1,2,\dotsc,t$. Thus, the number of white regular cells in the middle part of  $\mathcal{O}$ is given by
\begin{equation}\label{middnumber}
\mathcal{C}_2=\sum_{i=1}^{t}a_i(b_{i}+1).
\end{equation}
Similarly, the number of white regular cells in the middle part of $\mathcal{O}'$ is given by
\[\mathcal{C}'_2=\sum_{i=1}^{j-2}a_i(b_{i}+1)+(a_{j-1}+a_{j}+a_{j+1})(b_{j-1}+1)+\sum_{i=j+2}^ta_i(b_{i}+1).\]
Thus,
\begin{align}
\mathcal{C}_2-\mathcal{C}'_2&=a_{j-1}(b_{j-1}+1)+a_{j-1}(b_{j}+1)+a_{j-1}(b_{j+1}+1)\notag\\
&-(a_{j-1}+a_{j}+a_{j+1})(b_{j-1}+1)\notag\\
&=a_{j},
\end{align}
because we are assuming $b_{j-1}=b_{j}-1=b_{j+1}$. Therefore,
\begin{align}\label{oceq6}
&2^{\mathcal{C}'_1+\mathcal{C}'_2+\mathcal{C}'_3-h'_1(2w'-h'_1+1)/2-h'_2(2w'-h'_2+1)/2-h'_3(2w'-h'_3+1)/2}\notag\\
&\times2^{\binom{h'_1+h'_2}{2}-\binom{h'_2+h'_3}{3}}\mathcal{P}(h'_2+1,1,h'_1+h'_3,h'_1-h'_3,w'+h'_2-1)\notag\\
&=2^{\mathcal{C}_1+(\mathcal{C}_2-a_j)+\mathcal{C}_3-h_1(2w-h_1+1)/2-h_2(2w-h_2+1)/2-h_3(2w-h_3+1)/2}\notag\\
&\times2^{\binom{h_1+h_2}{2}-\binom{h_2+h_3}{3}}\mathcal{P}(h_2+1,1,h_1+h_3,h_1-h_3,w+h_2-1).
\end{align}
By the equalities (\ref{oceq5}) and (\ref{oceq6}), the statement of part (b) holds for $\mathcal{O}$ if and only if it holds for $\mathcal{O}'$.

\medskip

The case of concave $b_{j}$ is perfectly analogues to the case treated above. The only difference is that we use the transformation in Lemma \ref{OTrans}(b) (in reverse) instead of the transformation in Lemma \ref{OTrans}(a) (see Figures \ref{octa4}(c) and (d) for an example). The resulting region is still the quasi-octagon $\mathcal{O}'$ defined as in (\ref{newoc}); and by Lemma \ref{OTrans}(b), we have now
\begin{equation}\label{oceq7}
 \M(\mathcal{O})=2^{-a_{j}}\M(\mathcal{O}')
\end{equation}
 and
\begin{equation}\label{oceq8}
\mathcal{C}_2-\mathcal{C}'_2=-a_{j}.
\end{equation}
Similar to the case of convex $b_{j}$, the statements in parts (a) and (b) hold for $\mathcal{O}$ if and only if they hold for $\mathcal{O}'$.

\medskip

Keep applying this process if the resulting quasi-octagon still has more than one middle layer. Finally, we get a quasi-octagon $\widetilde{\mathcal{O}}$ with only one middle layer so that the statement of theorem holds for $\mathcal{O}$ if and only if it holds for $\widetilde{\mathcal{O}}$. It means that, without loss of generality, we can assume that $t=1$. This finishes Step 2.

\medskip

STEP 3. \textit{Prove the theorem for $k=l=t=1$.}\\

 We have in this case  $\mathcal{O}=\mathcal{O}^{(1)}_{w-1}(2h_1-1;\ 2h_2;\ 2h_3-1).$

 It is easy to see that the numbers of black cells and white cells must be the same if the region $\mathcal{O}$ admits tilings. One can check easily that the balancing condition between black and white cells in the region requires
\begin{equation}\label{balancingoc}
h_1+h_3=w+h_2.
\end{equation}
In particular, this implies the statement in part (a) of the theorem.

Assume that (\ref{balancingoc}) holds, and let
\begin{equation}
\mathcal{O}'':=\mathcal{O}^{(1)}_{w-h_1}(\underbrace{1,\dotsc,1}_{h_1};\ 2h_2;\ \underbrace{1,\dotsc, 1}_{h_3}).
\end{equation}
Apply the equality (\ref{octaeq1}) to the region $\mathcal{O}'$, we get
\begin{align}\label{octaeq3}
\M(\mathcal{O}'')&=2^{\mathcal{C}''_1-h''_1w''+\mathcal{C}''_3-h''_3w''}\M(\mathcal{O}^{(1)}_{w-1}(2h_1-1;\ 2h_2;\ 2h_3-1))\\
&=2^{\mathcal{C}''_1-h''_1w''+\mathcal{C}''_3-h''_3w''}\M(\mathcal{O}),
\end{align}
 where the double-primed symbols refer to the regions $\mathcal{O}''$ and denote quantities corresponding to their unprimed counterparts of $\mathcal{O}$.
One readily sees that $h''_1=h_1$, $w''=w$, and $h''_3=h_3$. Moreover, the numbers of black regular cells in the upper and lower parts of $\mathcal{O}''$ are given by
\begin{equation}
\mathcal{C}''_1= \sum_{i=0}^{h_1-1} (w-i)=h_1w-\frac{h_1(h_1-1)}{2}
\end{equation}
and
\begin{equation}
\mathcal{C}''_3=\sum_{i=0}^{h_3-1} (w-i)=h_3w-\frac{h_3(h_3-1)}{2}.
\end{equation}
Therefore, by (\ref{octaeq3}), we have
\begin{equation}\label{octaeq4}
\M(\mathcal{O}'')=2^{-\frac{h_1(h_1-1)}{2}-\frac{h_3(h_3-1)}{2}}\M(\mathcal{O})
\end{equation}
(see Figures \ref{octafinal}(a) and (b) for an example).

\begin{figure}\centering
\includegraphics[width=10cm]{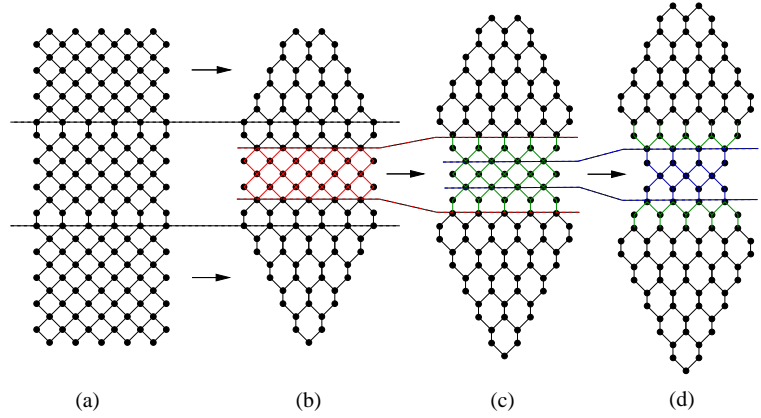}
\caption{Transforming process for the middle layer in the case when $k=l=t=1$.}
\label{octafinal}
\end{figure}

Let $G''$ be the dual graph of $\mathcal{O}''$. Consider the transforming process illustrated in Figures \ref{octafinal}(b)--(d) as follows.
Let $B_1$ is the subgraph consisting of all $h_2-1$ rows of $w-1$ diamonds in the middle part of $G''$, i.e. $B_1$ is isomorphic to $AR_{h_2-1,w-1}$ (see the graph between two inner dotted line in Figure \ref{octafinal}(b)). Apply the transformation in Lemma \ref{OTrans}(b) to replace $B_1$ by the graph ${}_|^|OR_{h_2-1,w-2}$ (see Figures \ref{octafinal}(b) and (c)). Next, consider the subgraph $B_2$ of the resulting graph that consists of all $h_2-2$ rows of $w-2$ diamonds of the graph, so $B_2$ is isomorphic to $AR_{h_2-2,w-2}$ (see the graph between two inner dotted lines in Figure \ref{octafinal}(c)). Apply the transformation in Lemma \ref{OTrans}(b) again to transform $B_2$ into ${}_|^|OR_{h_2-2,w-3}$. Keep applying the process until all rows of diamonds in the resulting graph have been eliminated (i.e. this process stops after $h_2-1$ steps).
%
%
 Denote by  $\widetilde{G}$ the final graph of the process, by Lemma \ref{OTrans}(b), we get
\begin{equation}\label{octaeq5}
\M(\mathcal{O}'')=\M(G'')=2^{\frac{h_2(h_2-1)}{2}}\M(\widetilde{G}).
\end{equation}
By (\ref{octaeq4}), we obtain
\begin{equation}\label{octaeq6}
\M(\mathcal{O})=2^{\frac{h_1(h_1-1)}{2}+\frac{h_2(h_2-1)}{2}+\frac{h_3(h_3-1)}{2}}M(\widetilde{G}).
\end{equation}

Moreover, $\widetilde{G}$ is exactly the graph $H_{h_1+h_2,w-h_1,h_2,h_2}^{h_2+h_3,w-h_3}$ in Lemma \ref{massholelem} (see Figure \ref{octafinal}(d)). By Lemma \ref{massholelem}, we have
\begin{align}\label{octaeq7}
\M(\widetilde{G})=&2^{-\frac{(h_1+h_2)(h_1+h_2-1)}{2}-\frac{(h_2+h_3)(h_2+h_3-1)}{2}}\notag\\
&\times \M(AR_{h_1+2h_2+h_3-1,w+h_2}^{|h_1-h_3|}(A)),
\end{align}
where $A=\{1,\dots,h_2\}\cup\{w+1,\dots,w+h_2\}$. Thus,  (\ref{mainoc'}) follows from the equality (\ref{octaeq7}) and Lemma \ref{schur3}. This finishes our proof for type-1 quasi-octagons.

\medskip

\quad\textit{Case 2. $\mathcal{O}$ is of type 2.}

\medskip

\begin{figure}\centering
\includegraphics[width=6cm]{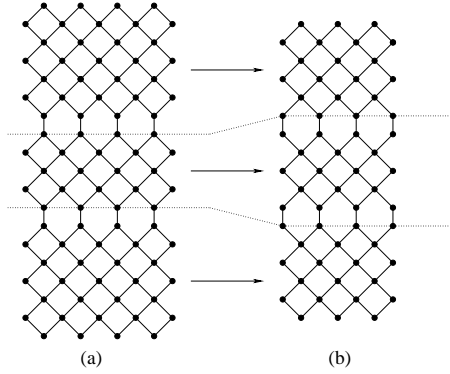}
\caption{Obtaining the dual graph of a type-1 quasi-octagon form the dual graph of a type-2 quasi-octagon.}
\label{octa2n}
\end{figure}

Repeat the argument in Steps 1 and 2 of Case 1, we can assume that $k=l=t=1$. We only need to prove the theorem for the type-2 quasi-octagon $\mathcal{O}:=\mathcal{O}^{(2)}_{w}(2h_1;\ 2h_2;\  2h_3)$. The dual graph $G$ of $\mathcal{O}$ can be divided into three subgraphs ${}_{|}AR_{h_1,w}$, ${}^{|}AR_{h_2,w}$, and $AR_{h_3,w}$ (in order from top to bottom) as in Figure \ref{octa2n}(a), for $w=4$, $h_1=3$, $h_2=2$ and $h_3=4$. Apply the transformation in Lemma \ref{T1} separately to replace ${}_{|}AR_{h_1,w}$ by the graph $AR_{h_1-\frac{1}{2},w-1}$, and ${}^{|}AR_{h_3,w}$ by the graph $AR_{h_3-\frac{1}{2},w-1}$ flipped over a horizontal line. Next, apply Lemma \ref{OTrans}(b) to transform $AR_{h_2,w}$ into ${}_{|}^{|}OR_{h_2,w-1}$. This way, the $G$ gets transformed in to the dual graph of the type-1 quasi-octagon $\mathcal{O}^{(1)}_{w-1}(2h_1-1;\ 2h_2;\  2h_3-1)$ (illustrated in Figure \ref{octa2n}(b)); and by Lemmas \ref{T1} and \ref{OTrans},
we obtain
\begin{equation}\label{oceq10}
\M(\mathcal{O})=2^{h_1+h_2+h_3}\M(\mathcal{O}^{(1)}_{w-1}(2h_1-1;\ 2h_2;\  2h_3-1)).
\end{equation}
Thus, both parts (a) and (b) of the theorem are reduced to the Case 1 treated above.

\medskip

\quad\textit{Case 3. $\mathcal{O}$ is of type 3.}

\medskip

Since $|BG|=|CF|$ and since $\mathcal{O}$ is of type 3, we have
\[-1=(w-1)-w=b_{1}-b_{t}=\sum_{i=1}^{t-1}(b_i-b_{i+1}).\] Since $|b_{i}-b_{i+1}|=1$, for $i=1,2,\dotsc,t-1$, the number of 1 terms and is one less than the number of $-1$ terms in the sequence $\{(b_i-b_{i+1})\}_{i=1}^{t-1}$. Thus, $t-1$ is odd, or $t$ is \textit{even} (as opposed to being odd in Case 1). By arguing similarly to Case 1, we can assume that $k=l=1$ and $t=2$.

The quasi-octagon $\mathcal{O}$ is now
\[\mathcal{O}^{(2)}_{w-1}(2h_1-1;\ 2x, 2y;\ 2h_3),\]
 where $x$ and $y$ are two positive integers such that $x+y=h_2$ (see Figure \ref{octa3}(a) for an example with $h_1=5$, $h_2=6$, $x=2$, $y=2$, $w=7$). Denote by $G$ the dual graph of $\mathcal{O}$  as usual. Apply Vertex-splitting Lemma to all topmost vertices of the lower part of $G$, and  divide the resulting graph by three horizontal dotted lines as in Figure \ref{octa3}(b). Apply the transformation in Lemma \ref{T1} to the bottom part, and the transformation in Lemma \ref{OTrans}(a) to the second part from top (sees Figures \ref{octa3}(b) and (c)). This way, $G$ is transformed into the dual graph of the type-1 quasi-octagon
 \[\mathcal{O}^{(1)}_{w-1}(2h_1-1;\ 2h_2;\ 2h_3-1)\]
  (see Figure \ref{octa3}(c)); and, by Lemmas \ref{T1} and \ref{OTrans}, we obtain
\begin{equation}\label{oceq11}
\M(\mathcal{O})=2^{h_2+h_3}\M(\mathcal{O}^{(1)}_{w-1}(2h_1-1;\ 2h_2;\ 2h_3-1)).
\end{equation}
Again, we get the statements of the theorem  from Case 1.

\begin{figure}\centering
\includegraphics[width=10cm]{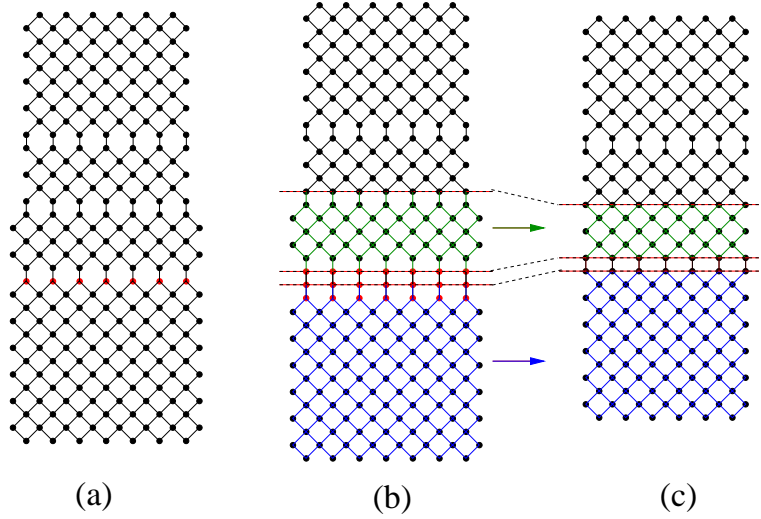}
\caption{Illustrating Case 3 of the proof of Theorem \ref{octagon1}.}
\label{octa3}
\end{figure}

\medskip

\quad\textit{Case 4. $\mathcal{O}$ is of type-4.}

\medskip

The type-3 quasi-octagon
\[\mathcal{O}^*:=\mathcal{O}^{(3)}_{|DE|}(d'_1,\dotsc,d'_l;\ \overline{d}_t,\dotsc, \overline{d}_1;\ d_1,\dotsc,d_k)\]
is obtained from our type-4 quasi-octagon
\[\mathcal{O}:=\mathcal{O}^{(4)}_{a}(d_1,\dotsc,d_k;\ \overline{d}_1,\dotsc, \overline{d}_t;\ d'_1,\dotsc,d'_l)\] by reflecting about $\ell$. Thus, this case follows from Case 3.
\end{proof}

Theorem \ref{octagon1} require $w> \max(h_1,h_2,h_3)$. The following theorem gives the number of tilings of a quasi-octagon when  $w\leq \max(h_1,h_2,h_3)$.

\begin{thm}\label{octagon2} Let $a$, $d_1,\dotsc,d_k;$ $\overline{d}_1,\dotsc, \overline{d}_t;$ $d'_1,\dotsc,d'_l$ be positive integers, so that for which the region $\mathcal{O}:=\mathcal{O}_{a}(d_1,\dotsc,d_k;\ \overline{d}_1,\dotsc, \overline{d}_t;\ d'_1,\dotsc,d'_l)$ is a quasi-octagon satisfying the  balancing condition (\ref{balancingoc}), having the heights $h_1,h_2,h_3$, and having both widths equal $w$. Let $\mathcal{C}_1$ be the numbers of black regular cells in the upper part, $\mathcal{C}_2$ be the number of white regular cells in the middle part, and $\mathcal{C}_3$ be the number of black regular cells in the lower part of the region.

(a) If $w<\max(h_1,h_2,h_3)$, then $\M(\mathcal{O})=0$.

(b) If $h_2=w$ (so, $h_1=h_2=h_3=w$ by (\ref{balancingoc})), then
\begin{equation}\label{specialoc1}
\M(\mathcal{O})=2^{\mathcal{C}_1+\mathcal{C}_2+\mathcal{C}_3-h_1(2w-h_1+1)/2-h_2(2w-h_2+1)/2-h_3(2w-h_3+1)/2}.
\end{equation}

(c) If $h_1= w$ and $h_2=h_3<w$, then
\begin{equation}\label{specialoc2}
\M(\mathcal{O})=2^{\mathcal{C}_1+\mathcal{C}_2+\mathcal{C}_3-h_1(2w-h_1+1)/2-h_2(2w-h_2+1)/2-h_3(2w-h_3+1)/2}\M(H_{h_2,w-h_2,h_2}).
\end{equation}

(d) The conclusion of part (2) is still true when $h_3= w$ and $h_1=h_2<w$.
\end{thm}

\begin{proof}
(a) By the same argument in Theorem \ref{octagon1}, we can assume that $k=l=t=1$ for type-1 and type-2 quasi-octagons, and $k=l=1$ and $t=2$ for type-3 and type-4 quasi-octagons. Then this part follows directly from Graph-splitting Lemma \ref{GS}, part (b).

(b) Suppose that $\mathcal{O}$ is of type-1. By the argument in Theorem \ref{octagon1}, we can assume that $k=l=t=1$ (then the general case can be obtained by induction on $t$). The quasi-octagon $\mathcal{O}$ is now $\mathcal{O}^{(1)}_{w-1}(2h_1-1;2h_2;\ 2h_3-1)$. Let $G$ be the dual graph of $\mathcal{O}$ as usual. 

Define $G_1$ to be the graph obtained from the upper part of $G$ by removing all bottom vertices, $G_2$ to be the graph obtained from the middle part of $G$ by removing all top and bottom vertices, and $G_3$ to be the graph obtained from the lower part of $G$ by removing all top vertices. Since $h_1=h_2=h_3=w$, the graphs $G_1$, $G_2$ and $G_3$ are isomorphic to the dual graph of the Aztec diamond of order $w-1$, and satisfy the condition in part (a) of the Graph-splitting Lemma. Thus,
\begin{align}
\M(G)&=\M(G_1)\M(G_2)\M(G_3)\M(G')\\
&=2^{3w(w-1)/2}\M(G'),
\end{align}
where $G'$ is the graph obtained from $G$ by removing $G_1$, $G_2$ and $G_3$. One readily sees that $G'$ consists of $2q$ disjoint vertical edges, so $\M(G')=1$. Moreover,  we have $h_1=h_2=h_3=w$ and $\mathcal{C}_1=\mathcal{C}_2=\mathcal{C}_3=w^2$, then thus
\begin{align}
\M(G)&=2^{3w(w-1)/2}\notag\\
&=2^{\mathcal{C}_1+\mathcal{C}_2+\mathcal{C}_3-h_1(2w-h_1+1)/2-h_2(2w-h_2+1)/2-h_3(2w-h_3+1)/2},
\end{align}
which implies (\ref{specialoc1}) and finishes our proof for type-1 quasi-octagons.

Finally, the equalities (\ref{oceq10}) and (\ref{oceq11}) in the proof of Theorem \ref{octagon1} are still true in this case. Thus, the case when $\mathcal{O}$ is of type 2, 3 or 4 can be reduced to the case treated above.

(c) Similar to part (b), we only need to consider the case when $\mathcal{O}$ is of type 1 and has $k=l=t=1$. Since $h_1=w$, the graph $G_1$ defined in part (b) is isomorphic to the dual graph of the Aztec diamond of order $w-1$ (so $\mathcal{C}_1=w^2$) and satisfies the two conditions in the part (a) of the Graph Splitting Lemma \ref{graphsplitting}, thus
\[\M(G)=2^{w(w-1)/2}\M(\overline{G})=2^{\mathcal{C}_1-h_1(2w-h_1+1)/2}\M(\overline{G}),\]
where $\overline{G}$ is the graph obtained from $G$ by removing $G_1$.  Remove all $w$ vertical  forced edges in top of $\overline{G}$, we get precisely the dual graph of the symmetric quasi-hexagon $H_{w-1}(2h_2-1;\ 2h_3-1)$ (see \cite{Tri}); and by Theorem 2.2  in \cite{Tri}, we have
\[M(\overline{G})=2^{\mathcal{C}_2-h_2(2w-h_2+1)/2+\mathcal{C}_3-h_3(2w-h_3+1)/2}\M(H_{h_2,w-h_2,h_2}).\]
This implies (\ref{specialoc2}).

(4) Part (4) can be reduced to part (3) by considering the region $\mathcal{O}':=\mathcal{O}^{(1)}_{w-1}(2h_3-1;2h_2;\ 2h_1-1)$ that is obtained by reflecting $\mathcal{O}=\mathcal{O}^{(1)}_{w-1}(2h_1-1;2h_2;\ 2h_3-1)$ over $\ell'$.
\end{proof}

As mentioned before, we do \textit{not} have a simple product formula for the number of tilings of a quasi-octagon when its widths are not equal, i.e. $w_1=|BG|\not=|CF|=w_2$. However, we have a sum formula for the number of tilings in this case.

Let $S=\{s_1,s_2,\dotsc,s_k\}$ be a set of integers, we define the operation \[\Delta (S):=\prod_{1\leq i<j\leq k}(s_j-s_i).\] For any positive integer $n$, we denote by $[n]$ the set the first $n$ positive integers $\{1,2,3,\dotsc,n\}$. Let $x$ be a number and $A$ be a set of numbers, we define $x+A:=\{y+x|y\in A\}$.

\begin{thm}\label{octagon3}
Let $a$, $d_1,$ $\dotsc$, $d_k$; $\overline{d}_1$, $\dotsc$, $\overline{d}_t$; $d'_1$, $\dotsc$, $d'_l$ be positive integers for which the region
\[\mathcal{O}b:=\mathcal{O}_{a}(d_1,\dotsc,d_k;\ \overline{d}_1,\dotsc, \overline{d}_t;\ d'_1,\dotsc,d'_l)\]
is a quasi-octagon having the upper, the middle, the lower heights $h_1,$ $h_2,$ $h_3$, respectively, and the upper and the lower widths $w_1,w_2$, respectively. Assume in addition that $w_1>w_2$, $h_1<w_1$, $h_2<w_2$, and $h_3<w_2$. Let $\mathcal{C}_1$ be the numbers of black regular cells in the upper part, $\mathcal{C}_2$ the number of white regular cells in the middle part, and $\mathcal{C}_3$ the number of black regular cells in the lower part of the region. Then

(a) If $h_1+h_3\not=w_1+h_2$, then $\M(\mathcal{O})=0$.

(b) Assume that $h_1+h_3=w_1+h_2$, then
\begin{align}\label{mainoc}
\M(\mathcal{O})&=2^{\mathcal{C}_1+\mathcal{C}_2+\mathcal{C}_3-h_1(2w-h_1+1)/2-h_2(2w-h_2+1)/2-h_3(2w-h_3+1)/2}\notag\\
&\times \sum_{(A,B)}\frac{\Delta([h_2+2w_1-w_2]\setminus (h_2+w_1-w_2+B))}{\Delta([h_1+h_2+w_1-w_2])}\frac{\Delta([w_2+h_2]\setminus (h_2+A))}{\Delta([h_2+h_3])},
\end{align}
where the sum is taken over all pairs of disjoint sets $A$ and $B$ so that $A\cup B=[w_2-h_2]$, $|A|=w_2-h_3$ and $|B|=w_1-h_1$.
\end{thm}

The following transformation can be proved similarly to Lemma \ref{OTrans}, and will be employed in the proof of Theorem \ref{octagon3}.
\begin{figure}\centering
\includegraphics[width=6cm]{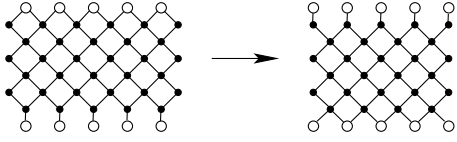}
\caption{Illustration of the transformation in Lemma \ref{T6}.}
\label{newoctatran}
\end{figure}

\begin{lem}\label{T6}
Let $G$ be a graph, and $p,q$ two positive integers. Assume that $\{v_1,v_2,$ $\dotsc,v_{2q}\}$ be an ordered set of vertices of $G$. Then
\begin{equation}
\M({}_|AR_{p,q} \#G)=2^{p}\M({}^|OR_{p,q-1} \#G),
\end{equation}
where ${}_|AR_{p,q}$ is defined as in Lemma \ref{T1}, and ${}^|OR_{p,q-1}$ is the graph obtained from $OR_{p,q-1}$ by appending $q$ vertical edges to its top vertices; and where the connected sum acts on $G$ along $\{v_1,v_2,\dotsc,v_{2q}\}$, and on ${}_|AR_{p,q}$ and ${}^|OR_{p,q-1}$ along their  $q$ top vertices ordered from left to right, then along their $q$ bottom vertices ordered  from left to right (see Figure \ref{newoctatran} for the case $p=2$ and $q=5$).
\end{lem}

\begin{proof}[Proof of Theorem \ref{octagon3}]
Similar to the proof of Theorem \ref{octagon1}, we can assume that the quasi-octagon $\mathcal{O}$ is of type 1 and has $k=l=1$.

Denote by $K_j$ the dual graph of the $j$-th middle layer of the region. Assume that the $j$-th middle layer has the height and the width $a_j$ and $b_j$, respectively.

 We apply the process in the proof of Theorem \ref{octagon1} (Case 1, Step 3) using the transformations in Lemma \ref{OTrans} to eliminate all the concave and convex terms in the sequence of the widths of the middle layers $\{b_{j}\}_1^t$ (see Figure \ref{octa4} for an example). It means that the sequence of the widths of the middle layers becomes a monotone sequence.  Since $w_1=b_{1}> b_{t}=w_2$, we can assume that $b_{1}>b_{2}>\dots>b_{t}$.

The dual graph $K_j$ of the $j$-th middle layer is  now isomorphic to the baseless Aztec rectangle  $AR_{a_j-\frac{1}{2},b_j}$ reflected  about its base. We consider a process using the transformation in Lemma \ref{T6} to the middle part of $G$  as follows.
\begin{figure}\centering
\includegraphics[width=10cm]{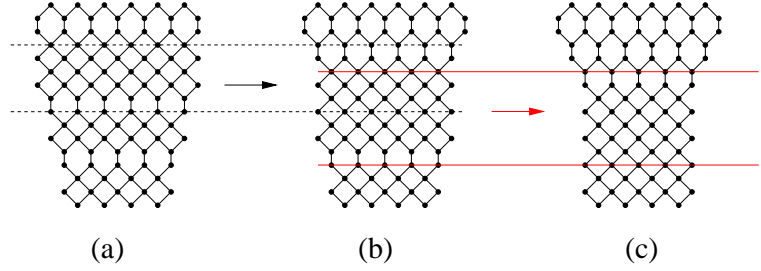}
\caption{Transforming process for the middle parts of a quasi-octagon when $w_1>w_2$.}
\label{process6}
\end{figure}

Assume that there  exists a middle layer other than the last one which has positive height. Assume that the $j_0$-th middle layer is the first such layer. Consider the graph $Q_1$ obtained from $K_{j_0}$ by removing all its top vertices and appending vertical edges to its bottom vertices.  Apply the transformation in Lemma \ref{T6} to replace $Q_1$ by the graph ${}^|OR_{a_{j_0},b_{j_0}-1}$. This transformed $G$ into the dual graph of new quasi-octagon $\mathcal{O}'$, which has the same upper and lower parts as $\mathcal{O}$, and the sequence of sizes of the middle layers
\[\big((0,b_1), (0,b_2), \dotsc,(0,b_{j_0}), (a_{j_0}+a_{j_0+1},b_{j_0+1}),(a_{j_0+2},b_{j_0+2}),\dotsc,(a_t,b_t)\big).\]
This step is illustrated in Figures \ref{process6}(a) and (b), the subgraph between two dotted lines in Figure \ref{process6}(a) is replaced by the one between these two lines in Figure \ref{process6}(b). Apply again the transformation in Lemma \ref{T6} to the graph $Q_2$ obtained from the dual graph of the $(j_0+1)$-th middle layer $\mathcal{O}'$ by removing the top vertices and appending $b_{j_0+1}$ vertical edges to bottom (see Figures \ref{process6}(b) and (c)), and so on. The procedure stops when the heights of all middle layers in the resulting region, except for the last one, are equal to $0$. Denote by $\mathcal{O}^*$ the final quasi-octagon, so $\mathcal{O}^*$ has the sequence of sizes of the middle layers
\[\big((0,b_1), (0,b_2),\dotsc,(0,b_{t-1}),(h_2,b_t)\big).\]

\begin{figure}\centering
\includegraphics[width=12cm]{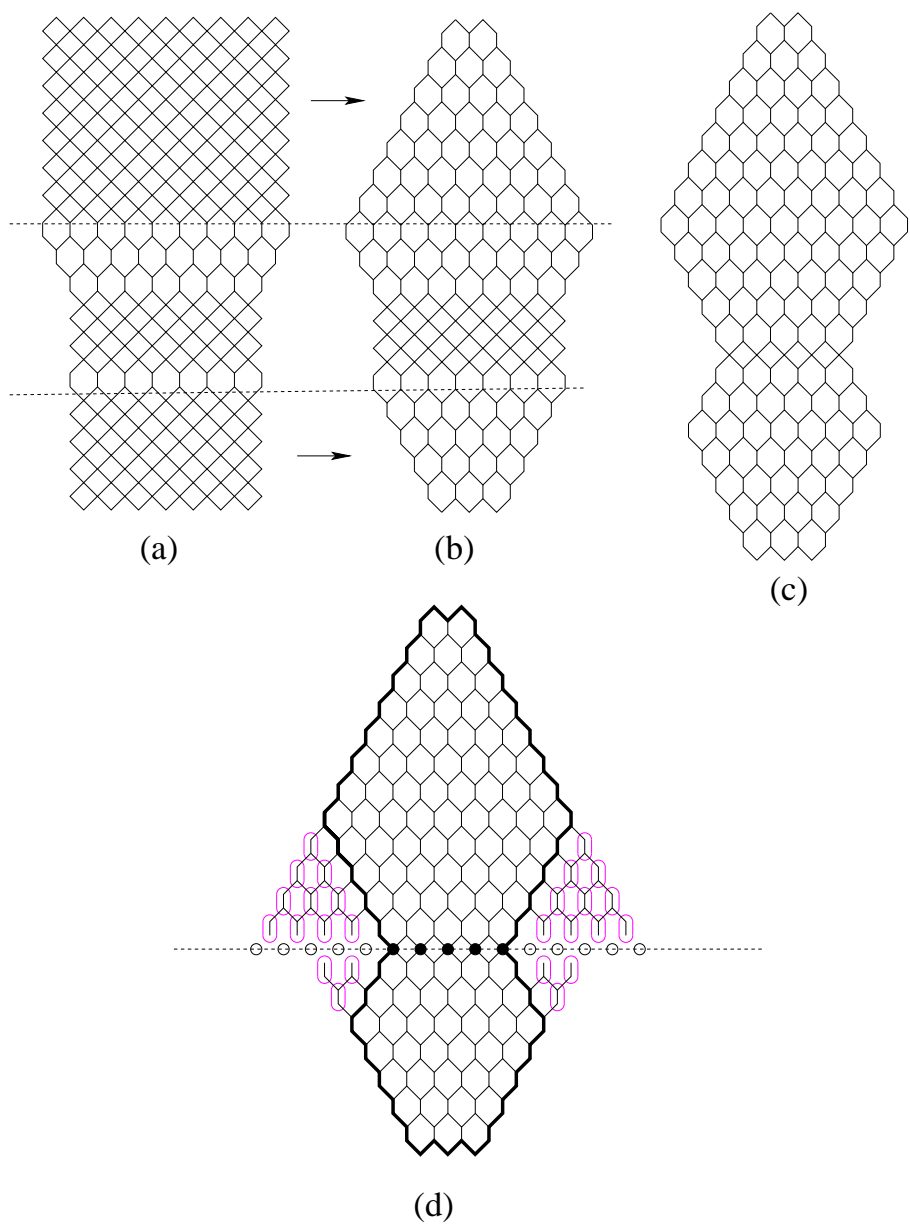}
\caption{Illustrating the proof of Theorem \ref{octagon3}.}
\label{generalocta}
\end{figure}

One readily sees that the above process preserves the heights, the widths, and the lower and upper parts of the quasi-octagon. Moreover, by Lemma \ref{T6} and the equality (\ref{middnumber}) in the proof of Theorem \ref{octagon1}, we get
\begin{align}
\M(\mathcal{O})/\M(\mathcal{O}^*)&=2^{a_{j_0}+(a_{j_0}+a_{j_0+1})+\dotsc+(a_{j_0}+\dotsc+a_{i-1})}\notag\\
&=2^{\mathcal{C}_2}/2^{\mathcal{C}^*_2}\notag\\
&=\mathcal{Q}(h_1,h_2,h_3,w_1,w_2,\mathcal{C}_1,\mathcal{C}_2,\mathcal{C}_3)/\mathcal{Q}(h^*_1,h^*_2,h^*_3,w^*_1,w^*_2,\mathcal{C}^*_1,\mathcal{C}^*_2,\mathcal{C}^*_3),
\end{align}
where the star symbols refer to the region $\mathcal{O}^*$ and denote quantities corresponding to their non-starred counterparts of $\mathcal{O}$, and where $\mathcal{Q}(h_1,h_2,h_3,w_1,w_2,\mathcal{C}_1,\mathcal{C}_2,\mathcal{C}_3)$ denotes the expression on the right hand side of (\ref{mainoc}). This implies that the statement of the theorem is true for $\mathcal{O}$ if and only its is true for $\mathcal{O}^*$. It means that we can assume that $a_{j}=0$, for all $1\leq j< t$.

\medskip

Next, we prove the theorem for the case when $k=l=1$ and $a_{j}=0$, for $1\leq j< t$ (illustrated in Figure \ref{generalocta}).

Similar to  the proof of Theorem \ref{octagon1} (Case 1, Step 3), we apply the transformation in Lemma \ref{masstransform2}(2) (in reverse) to transform the upper and lower parts of the dual graph $G$ of $\mathcal{O}$ into the dual graphs of two semi-hexagons (see Figures \ref{generalocta}(a) and (b)), and the transformation in Lemma \ref{OTrans}(2) to transform the dual graph $K_t$ of the last middle layer into a butterfly-shaped graph (see Figures \ref{generalocta}(b) and (c)). This way, $G$ gets transformed into the graph $\overline{G}=H_1\# H_2$, where $H_1$ is the dual graph of a hexagon of sides $w_1-h_1, h_1,h_2+w_1-w_2,w_2-h_2,h_2+w_1-w_2,h_1$ and  $H_2$ is the dual graph of a hexagon of sides $w_2-h_2, h_2,h_3,w_2-h_3,h_3,h_2$ (in cyclic order starting by the north side)  on the triangular lattice (see Figures \ref{generalocta}(a) and (c)). By Lemmas \ref{masstransform2} and \ref{OTrans}, we get
\begin{align}\label{octagong}
\M(\mathcal{O})&=2^{h_1(h_1-1)/2+h_2(h_2-1)/2+h_3(h_3-1)/2}\M(\overline{G})\notag\\
&=2^{\mathcal{C}_1+\mathcal{C}_2+\mathcal{C}_3-h_1(2w-h_1+1)/2-h_2(2w-h_2+1)/2-h_3(2w-h_3+1)/2}\M(\overline{G}).
\end{align}

The graph $\overline{G}$ is in turn  obtained from $\widetilde{G}:=S_1\#S_2$ by removing the vertical forced edges (shown by the circled edges in Figure \ref{generalocta}(d)), where $S_1$ is the dual graph of the $(h_1+h_2+w_1-w_2,w_1-h_1)$-semihexagon with $h_2+w_1-w_2$ leftmost and  $h_2+w_1-w_2$ rightmost bottom vertices removed, and $S_2$ is the dual graph of $(h_2+h_3,w_2-h_3)$-semihexagon with $h_2$ leftmost and $h_2$ rightmost bottom vertices removed; and where the connected sum acts on $S_1$ and $S_2$ along their bottommost vertices ordered from left to right (see Figure \ref{generalocta}(d)). Since removing forced edges does not change the number of perfect matchings of a graph, $\M(\overline{G})=\M(\widetilde{G})$.

There are $w_2-h_2$ vertices belonging to both $S_1$ and $S_2$; and we partition the set of perfect matchings of $\widetilde{G}$ into $2^{w_2-h_2}$ classes corresponding to all the possible choices for each of these vertices to be matched upward or downward. Each class is then the set of perfect matchings of a disjoint union of two graphs, being of the kind in Lemma \ref{Hexdent}. Part (a) follows from the requirement that the graph $\widetilde{G}$ has the numbers of vertices in two vertex classes equal, while part (b) follows from Lemma \ref{Hexdent}.
\end{proof}

\thispagestyle{headings}

\end{document}